\documentclass[a4paper,12pt]{article}
\usepackage[left=2cm,right=2cm, top=2cm,bottom=3cm,bindingoffset=0cm]{geometry}

\usepackage{verbatim}
\usepackage{amsmath}
\usepackage{amsthm}
\usepackage{amssymb}
\usepackage{delarray}
\usepackage{cite}
\usepackage{hyperref}
\usepackage{mathrsfs}
\usepackage{tikz}
\usetikzlibrary{patterns}
\usepackage{caption}
\newcommand{\la}{\lambda}

\DeclareCaptionLabelSeparator{dot}{. }
\theoremstyle{plain}

\numberwithin{equation}{section}

\newtheorem{thm}{Theorem}[section]
\newtheorem{lem}[thm]{Lemma}
\newtheorem{prop}[thm]{Proposition}
\newtheorem{cor}[thm]{Corollary}

\theoremstyle{definition}

\newtheorem{alg}[thm]{Algorithm}
\newtheorem{ip}[thm]{Inverse Problem}

\theoremstyle{remark}

\newtheorem{remark}[thm]{Remark}

 \allowdisplaybreaks

\begin{document}
\begin{center}
{\large\bf Solving the inverse Sturm-Liouville problem with singular potential\\[0.2cm] and with polynomials in the boundary conditions}
\\[0.2cm]
{\bf Egor E. Chitorkin*, Natalia P. Bondarenko} \\[0.2cm]
\end{center}

\vspace{0.5cm}

{\bf Abstract.}  In this paper, we for the first time get constructive solution for the inverse Sturm-Liouville problem with complex-valued singular potential and with polynomials of the spectral parameter in the boundary conditions. 
The uniqueness of recovering the potential and the polynomials from the Weyl function is proved. An algorithm of solving the inverse problem is obtained and justified.
More concretely, we reduce the nonlinear inverse problem to a linear equation in the Banach space of bounded infinite sequences and then derive reconstruction formulas for the problem coefficients, which are new even for the case of regular potential.
Note that the spectral problem in this paper is investigated in the general non-self-adjoint form, and our method does not require the simplicity of the spectrum. In the future, our results can be applied to investigation of the inverse problem solvability and stability as well as to development of numerical methods for the reconstruction.

\medskip


\medskip


\vspace{1cm}

\section{Introduction}

In this paper, we consider the Sturm-Liouville differential expression $\mathit{l}y = -y'' + q(x)y$ with a singular complex-valued potential $q \in W^{-1}_2(0, \pi)$. This means that $q = \sigma'$, where $\sigma \in L_2(0, \pi)$, and the derivative is understood in the sense of distributions. Then, the differential expression $\mathit{l}y$ can be represented in the following form:
\begin{gather*}
\mathit{l}y = -(y^{[1]})' - \sigma(x)y^{[1]} - \sigma^2(x)y,
\end{gather*}
where $y^{[1]}(x) := y'(x) - \sigma(x)y(x)$ is the so-called $\it{quasi}$-$\it{derivative}$.

We study the inverse spectral problem for the Sturm-Liouville equation
\begin{equation} \label{eqv}
l y = \lambda y, \quad x \in (0, \pi),
\end{equation}
with the boundary conditions (BCs)
\begin{equation} \label{bc}
p_1(\la)y^{[1]}(0) - p_2(\la)y(0) = 0, \quad r_1(\la)y^{[1]}(\pi) + r_2(\la)y(\pi) = 0,
\end{equation}
where $\la$ is the spectral parameter and $p_1(\la)$, $p_2(\la)$, $r_1(\la)$, and $r_2(\la)$ are polynomials.

Inverse problems of spectral analysis consist in the recovery of differential operator coefficients from spectral information. Such problems have applications in quantum mechanics, geophysics, chemistry, electronics, and other branches of science and engineering. Inverse spectral theory for the Sturm-Liouville operators with \textit{constant} coefficients in the BCs has been developed fairly completely (see the monographs \cite{Mar77, Lev84, PT87, FY01}). There is also a number of studies concerning eigenvalue problems with \textit{polynomial} dependence on the spectral parameter in the BCs. The latter problems arise in various physical applications, e.g., in mechanical engineering, diffusion, and electric circuit problems (see \cite{Ful77, Ful80} and references therein). 

\textit{Inverse} Sturm-Liouville problems with polynomials in the BCs have been studied in \cite{Chu01, BindBr021, BindBr022, BindBr04, ChFr, FrYu, FrYu12, Wang12, YangXu15, YangWei18, Gul19, Gul20-ann, Gul20-ams, Gul23, Chit, PST23} and other papers. The majority of the studies in this direction deal with self-adjoint problems containing rational Herglotz-Nevanlinna functions of the spectral parameter in the BCs (see, e.g., \cite{BindBr021, BindBr022, YangWei18, Gul19, Gul20-ann, Gul20-ams, Gul23}). It is easy to check that the BCs of that type can be reduced to the form with polynomial dependence on the spectral parameter. 

A constructive solution of the inverse Sturm-Liouville problem on a finite interval with the polynomial BCs in the general non-self-adjoint form has been obtained by Freiling and Yurko \cite{FrYu} by using the method of spectral mappings. It is a universal method which allows to reduce various types of inverse spectral problems to linear equations in suitable Banach spaces. By using this method, a number of results were obtained for higher-order differential operators, differential systems, quadratic Sturm-Liouville pencils, differential operators on graphs, etc. (see \cite{Yur02}). However, the application of the method of spectral mappings to different operator classes has its own features and so requires a separate investigation. In particular, an inverse problem with polynomials in the BC on the half-line was studied in \cite{FrYu12}. We point out that the results of Freiling and Yurko \cite{FrYu, FrYu12} are concerned only with regular (integrable) potentials.

In recent years, spectral analysis of differential operators with singular (distributional) coefficients has attracted much attention of mathematicians (see \cite{Gul19, SavShkal, SavShkal03, Hry03, Hry04, Sav05, Sav08, Sav10, Mirz, MirzShkal, Kon, Hry,  FrIgYu, Dj, BondTamkang}). Properties of spectral characteristics and solutions of differential equations with singular coefficients were studied in such papers as \cite{SavShkal, SavShkal03, Hry03, Hry04, Sav05, Sav08, Sav10, Mirz, MirzShkal, Kon}. The most complete results in the inverse problem theory have been obtained for the Sturm-Liouville operators with singular potentials (see \cite{Hry03, Hry04, Hry, Sav05, Sav08, Sav10, FrIgYu, Dj, BondTamkang}). In particular, Hryniv and Mykytyuk \cite{Hry03, Hry04, Hry} generalized a number of classical results to the case of $W_2^{-1}$-potentials, by using the transformation operator method. Savchuk and Shkalikov \cite{Sav10} proved the uniform stability of the inverse Sturm-Liouville problems with potentials of $W_2^{\alpha}$, where $\alpha > -1$. The method of spectral mappings has been transferred to the Sturm-Liouville operators with potentials of the class $W_2^{-1}$ in \cite{FrIgYu} and \cite{BondTamkang}.
Note that this class of distributions contains the Dirac $\delta$-function and the Coulomb potential $\frac{1}{x}$, which are used for modeling particle interactions in quantum mechanics (see \cite{Alb05}). There were also studies of inverse problems with singular potentials for the quadratic Sturm-Liouville pencils (see \cite{Hry12, Pr13, Hry20, BondGaidel, Kuz23}) and for the matrix Sturm-Liouville operators (see \cite{Myk09, Eck15, BondAMP21}).

Recently, the spectral data characterization for the self-adjoint Sturm-Liouville operators with potentials of $W_2^{-1}(0,\pi)$ and with Herglotz-Nevanlinna functions of $\la$ in the BCs has been obtained by Guliyev \cite{Gul19}. The method of \cite{Gul19} is based on the Darboux transformations and requires the self-adjointness of the operators. The non-self-adjoint inverse Sturm-Liouville problems with singular potentials and with polynomials in the BCs, to the best of the authors' knowledge, have not been investigated before. This paper aims to fill this gap.

Thus, this paper is concerned with the development of the inverse spectral theory for the Sturm-Liouville problem \eqref{eqv}-\eqref{bc} with the polynomial dependence on $\la$ in the BCs and with singular potential of class $W_2^{-1}(0,\pi)$. We consider an inverse problem that consists in the recovery of the function $\sigma(x)$ and of the polynomials $r_j(\la)$, $j = 1, 2$, from the Weyl function. The polynomials $p_1(\la)$ and $p_2(\la)$ are supposed to be known a priori. We prove the uniqueness theorem for this inverse problem and develop a constructive algorithm for its solution. Namely, we reduce the nonlinear inverse problem to a linear equation in the Banach space of bounded infinite sequences and obtain reconstruction formulas for the problem coefficients $\sigma$, $r_1$, and $r_2$ in form of series. The validity of the reconstruction formulas is rigorously proved by using the approximation technique.
Note that we consider the problem in the general non-self-adjoint form, and our method does not require the simplicity of the spectrum. 

Our approach develops the method of spectral mappings and relies on the ideas of \cite{FrYu} on a constructive solution of the inverse Sturm-Liouville problem with polynomials in the BCs and with the regular potential $q \in L_2(0,\pi)$, of the studies \cite{But07, But13} for dealing with multiple eigenvalues and of the paper \cite{BondTamkang} concerning with singular potentials. Note that, comparing to the solution method of Freiling and Yurko \cite{FrYu}, our method is easier. At first, we reduce the inverse problem for \eqref{eqv}-\eqref{bc} to the problem with $p_1(\la) \equiv 1$, $p_2(\la) \equiv 0$, and this simplifies the technique of the next steps. Moreover, we obtain the reconstruction formulas, which are novel even for the case of regular $L_2$-potential. These formulas will be convenient for the future study of solvability and stability of the inverse problem. Furthermore, relying on our theoretical algorithm for solving the inverse problem, one can develop a numerical method of the reconstruction. We mention that a numerical solution of the simplest Sturm-Liouville problem with constant BC coefficients by the method of spectral mappings has been obtained by Ignatiev and Yurko~\cite{IgYu}. Some other numerical techniques for solving inverse Sturm-Liouville problems can be found in \cite{Ru92, Krav20}. It is also worth mentioning that, in recent years, inverse spectral problems are being actively studied for nonlocal operators (see, e.g., \cite{BB20, But23, WKS23} and references therein), but, for such operators, essentially different methods are needed.

The paper is organized as follows. In Section~\ref{sec:main}, the inverse problem statements and the main results are formulated. In Section~\ref{sec:uniq}, we prove the uniqueness theorem. In Section~\ref{sec:maineq}, we derive the main equation. In Section~\ref{sec:solution}, we prove a theorem on the reconstruction formulas. In Section~\ref{sec:solution_alg}, we obtain an algorithm for solving the inverse problem and give a simple example of recovering the potential and the polynomials. In Section~\ref{sec:reg}, we apply our results to the inverse problem with regular potential from \cite{FrYu}.

\section{Main results and methods} \label{sec:main}

Denote by $L = L(\sigma, p_1, p_2, r_1, r_2)$ the boundary value problem \eqref{eqv}-\eqref{bc}, where $\sigma(x)$ is a complex-valued function of $L_2(0,\pi)$, $p_1(\la)$ and $p_2(\la)$ are relatively prime polynomials, and so do $r_1(\la)$ and $r_2(\la)$.
Obviously, the polynomials of the boundary conditions can be represented in the form
\begin{gather} \label{polsp}
p_1(\la) = 	\sum \limits_{n=0}^{N_1} a_n \la^n, \quad  p_2(\la) = \sum \limits_{n=0}^{N_2} b_n \la^n, \quad N_1, N_2 \ge 0, \\ \label{polsr}
r_1(\la) = 	\sum \limits_{n=0}^{M_1} c_n \la^n, \quad  r_2(\la) = \sum \limits_{n=0}^{M_2} d_n \la^n, \quad M_1, M_2 \ge 0.
\end{gather}

Denote by $\psi(x, \la)$ the solution of equation \eqref{eqv} satisfying the initial conditions $\psi(\pi, \la) = r_1(\la)$, $\psi^{[1]}(\pi, \la) = -r_2(\la)$.
Let us introduce the Weyl function of the problem $L$ as follows:
\begin{gather} \label{defM1}
M^1(\la) := -\frac{\psi(0,\la)}{p_1(\la) \psi^{[1]}(0,\la) - p_2(\la) \psi(0,\la)}.
\end{gather}

Note that the functions $\psi(x, \la)$ and $\psi^{[1]}(x, \la)$ are entire in $\la$ for each fixed $x \in [0,\pi]$, so $M^1(\la)$ is meromorphic in $\la$ with a countable set of poles at the zeros of the denominator in \eqref{defM1}.

Consider the following inverse problem.

\begin{ip} \label{ip:main}
	Given the Weyl function $M^1(\la)$ and the polynomials $p_1(\la)$, $p_2(\la)$, find $q(x)$, $r_1(\la)$, and $r_2(\la)$.
\end{ip}

Note that Inverse Problem~\ref{ip:main} can be reduced to the problem without the polynomials $p_1(\la)$ and $p_2(\la)$. Indeed, let us introduce the boundary value problem $\mathcal{L} = \mathcal{L}(\sigma, r_1, r_2)$:
\begin{gather*}
\mathit{l}y(x) = \la y(x), \quad x \in (0, \pi), \\
y^{[1]}(0) = 0, \quad r_1(\la)y^{[1]}(\pi) + r_2(\la)y(\pi) = 0.
\end{gather*}

The corresponding Weyl function has the form 
\begin{equation} \label{defM}
M(\la) = -\dfrac{\psi(0,\la)}{\psi^{[1]}(0,\la)}.
\end{equation}

\begin{ip} \label{ip:main2}
	Given the Weyl function $M(\la)$, find $q(x)$, $r_1(\la)$, and $r_2(\la)$.
\end{ip}

It follows from \eqref{defM1} and \eqref{defM} that
\begin{gather} \label{pass}
M(\la) = \frac{p_1(\la)M^1(\la)}{1+p_2(\la)M^1(\la)}.
\end{gather}

Although the denominator of \eqref{pass} can have a countable set of zeros, the function $M(\la)$ is uniquely specified by \eqref{pass} in the other points of the complex plane. So, $M(\la)$ can be analytically continued to all $\la$ different from its poles.
Hence, without loss of generality we can pass from Inverse Problem~\ref{ip:main} to Inverse Problem~\ref{ip:main2}. Therefore, below we focus on investigation of Inverse Problem~\ref{ip:main2}.

Let us normalize the polynomials $r_1(\la)$ and $r_2(\la)$. If $\deg (r_1) \ge \deg(r_2)$, then without loss of generality we assume that $c_{M_1} = 1$ and $M_1 = M_2$ in \eqref{polsr}. The coefficient $d_{M_2}$ can be equal to zero (so we have the polynomial $r_2(\la)$ of lesser degree). If $\deg(r_1) < \deg(r_2)$, then without loss of generality we assume that $d_{M_2} = 1$ and $M_1 = M_2 - 1$. In this case, $c_{M_1}$ can be equal to zero (so we have the polynomial $r_1(\la)$ of lesser degree). We will write that $(r_1, r_2) \in \mathcal N$ if the polynomials $r_1(\la)$ and $r_2(\la)$ fulfill these conditions. 

Along with the problem $\mathcal{L} = \mathcal{L}(\sigma, r_1, r_2)$, we consider another problem $\tilde {\mathcal{L}} = \mathcal{L}(\tilde \sigma, \tilde r_1, \tilde r_2)$ of the same form as $\mathcal{L}$, but with different potential $\tilde q$ and polynomials $\tilde r_1$, $\tilde r_2$. We agree that, if a certain symbol $\gamma$ denotes an object related to $\mathcal{L}$, then the symbol $\tilde \gamma$ with tilde will denote the analogous object related to $\tilde {\mathcal{L}}$. Note that the quasi-derivatives $y^{[1]} = y' - \sigma(x) y$ and $y^{[1]} = y' - \tilde \sigma(x) y$ corresponding to the problems $\mathcal L$ and $\tilde {\mathcal L}$, respectively, differ from each other. In addition, we assume that $(r_1, r_2) \in \mathcal N$ and $(\tilde r_1, \tilde r_2) \in \mathcal N$.

The first result of this paper is the following uniqueness theorem for Inverse Problem~\ref{ip:main2}.

\begin{thm} \label{thm:uniq}
	If $M(\la) = \tilde M(\la)$, then $\sigma(x) = \tilde\sigma(x)$ a.e. on $(0,\pi)$ and $r_j(\la) \equiv \tilde r_j(\la)$, $j=\overline{1, 2}$.
\end{thm}

The proof of Theorem~\ref{thm:uniq} is based on the method of spectral mappings (see \cite{FY01, BondTamkang}). Obviously, Theorem~\ref{thm:uniq} implies the following corollary on the uniqueness of solution for Inverse Problem~\ref{ip:main2}.

\begin{cor}
	If $M^1(\la) = \tilde M^1(\la)$ and $p_j(\la) \equiv \tilde p_j(\la)$, $j=\overline{1, 2}$, then $\sigma(x) = \tilde\sigma(x)$ a.e. on $(0,\pi)$ and $r_j(\la) \equiv \tilde r_j(\la)$, $j=\overline{1, 2}$.
\end{cor}

For constructive solution of the inverse problem, it is more convenient to use discrete spectral data, than the Weyl function. The spectrum of the problem $\mathcal L$ is described by the following lemma.

\begin{lem} \label{eigen_asymp}
	The spectrum of the boundary value problem $\mathcal L$ is a countable set of eigenvalues which can be numbered as $\{\la_n\}_{n \ge 1}$ (counting with multiplicities) according to their asymptotic behavior:
	\begin{gather}\label{la_asymp}
	\rho_n = \sqrt{\la_n} = 
	\left\{\begin{array}{ll}
	n - M_1 - 1 + \varkappa_n, \quad & M_1 \ge M_2, \\
	n - M_2 - \dfrac{1}{2} + \varkappa_n, \quad &  M_1 < M_2,
	\end{array}\right.
	\end{gather}
	where $\{ \varkappa_n \} \in l_2$.
\end{lem}

Note that the sequence $\{\la_n\}_{n \ge 1}$ can contain a finite number of multiple eigenvalues. Without loss of generality, we assume that multiple eigenvalues are consecutive, that is, $\la_k = \la_{k+1} = \dots = \la_{k+m_k-1}$, where $m_k$ is the multiplicity of $\la_k$. Put $I := \{1\}\cup\{n > 1 \colon \la_{n-1} \neq \la_n\}$.

Similarly to Lemma~2 in \cite{FrYu}, one can prove the following lemma.

\begin{lem} \label {weyl_func}
	The following representation holds:
	\begin{gather}\label{weyl_sum}
	M(\la) = \sum_{k \in I}^{} \sum_{j=0}^{m_k - 1} \frac{\alpha_{k+j}}{(\la - \la_k)^{j+1}},
	\end{gather}
	where $\alpha_{k+j}$ are the coefficients of the principal part of the Weyl function $M(\la)$ in a neighbourhood of the pole $\la_k$. The series \eqref{weyl_func} converges absolutely and uniformly on compact sets without eigenvalues of $\mathcal L$.
\end{lem}

In view of the eigenvalue asymptotics \eqref{la_asymp}, the poles of $M(\la)$ for sufficiently large $|\la|$ are simple. Thus, in formula \eqref{weyl_sum}, $m_k = 1$ for sufficiently large values of $k$.

We call the numbers $\{\la_n, \alpha_n\}_{n \ge 1}$ \textit{the spectral data} of the problem $\mathcal L$. Due to the formula \eqref{weyl_sum}, the spectral data uniquely determine the Weyl function and vice versa. Thus, instead of Inverse Problem~\ref{ip:main2}, one can study the following problem.

\begin{ip}\label{ip:main3}
Given the spectral data $\{ \la_n, \alpha_n \}_{n \ge 1}$, find $\sigma(x)$, $r_1(\la)$, and $r_2(\la)$.
\end{ip}

For solving Inverse Problem~\ref{ip:main3}, we develop the ideas of the method of spectral mappings for the non-self-adjoint Sturm-Liouville operator (see \cite{FY01, BondTamkang, But07}). Let us briefly describe the solution. For definiteness, we consider the case $M_1 = M_2$. The other case is analogous. Along with the problem $\mathcal L = \mathcal L(\sigma, r_1, r_2)$, we consider \textit{the model problem} $\tilde {\mathcal L} = \mathcal L(0, \la^{M_1}, 0)$. Denote by $\varphi(x, \la)$ the solution of equation \eqref{eqv} satisfying the initial conditions $\varphi(0,\la) = 1$, $\varphi^{[1]}(0,\la) = 0$. 

Introduce the notations
\begin{equation} \label{not}
\def\arraystretch{1.5}
\left.
\begin{array}{c}
\la_{n0} = \la_n, \quad \la_{n1} = \tilde \la_n, \quad \rho_{n0} = \rho_n, \quad \rho_{n1} = \tilde \rho_n, \quad \alpha_{n0} = \alpha_n, \quad \alpha_{n1} = \tilde \alpha_n, \\
\quad I_0 = I, \quad I_1 = \tilde I, \quad m_{n0} = m_n, \quad m_{n1} = \tilde m_n, \quad 
\varphi_{j}(x, \la) = \dfrac{1}{j!} \dfrac{\partial^j}{\partial \la^j} \varphi(x, \la), \\
\varphi_{n+j, i}(x) = \varphi_j(x, \la_{ni}), \quad \tilde \varphi_{n+j, i}(x) = \tilde \varphi_j(x, \la_{ni}), \quad n \in I_i, \quad j = \overline{0, m_{ki}}.
\end{array} \quad \right\}
\end{equation}

By using the contour integration in the $\la$-plane, we derive the following relations (see Lemma~\ref{A}):
\begin{equation} \label{series}
\tilde \varphi_{n,i}(x) = \varphi_{n,i}(x) + \sum_{k = 1}^{\infty} (\tilde Q_{n,i; k,0}(x) \varphi_{k,0}(x) - \tilde Q_{n,i; k,1}(x) \varphi_{k,1}(x)), \quad n \ge 1, \, i = 0, 1,
\end{equation}
where $\tilde Q_{n,i; k,j}(x)$ are some functions which are constructed by using only the model problem $\tilde{\mathcal L}$ and the spectral data $\{ \la_n, \alpha_n\}_{n \ge 1}$ of $\mathcal L$. 

The relations \eqref{series} can be treated as an infinite linear system with respect to the unknown values $\{ \varphi_{n,i}(x)\}_{n \ge 1, \, i = 0, 1}$ for each fixed $x \in [0, \pi]$. However, it is inconvenient to use \eqref{series} for investigation of the inverse problem solvability, because the series in \eqref{series} converges only ``with brackets''. In order to achieve the absolute convergence, we transform \eqref{series} to a linear equation in a suitable Banach space.

Let $m$ be the Banach space of bounded infinite sequences $a = (a_{ni})_{n \ge 1, \, i = 0, 1}$ with the norm $\| a \| = \sup\limits_{n,i} |a_{ni}|$. In Section~\ref{sec:maineq}, the system \eqref{series} is reduced to the so-called \textit{main equation} in $m$ of form
\begin{equation} \label{main-eq}
\tilde \psi(x) = (E + \tilde H(x)) \psi(x), \quad x \in [0,\pi],
\end{equation}
where $\psi(x)$ and $\tilde \psi(x)$ are sequences of $m$, $\tilde H(x)$ is a compact operator in $m$ for each fixed $x \in [0,\pi]$, and $E$ is the identity operator in $m$. The sequences $\psi(x)$, $\tilde \psi(x)$ and the operator $\tilde H(x)$ are obtained by grouping the terms $\varphi_{n,i}(x)$, $\tilde \varphi_{n,i}(x)$, and $\tilde Q_{n,i; k,j}(x)$, respectively. Thus, $\tilde \psi(x)$ and $\tilde H(x)$ are constructed by the model problem $\tilde{\mathcal L}$ and the spectral data $\{ \la_n, \alpha_n\}_{n \ge 1}$, while the sequence $\psi(x)$ is related to the unknown coefficients $\sigma(x)$, $r_1(\la)$, and $r_2(\la)$ of the problem $\mathcal L$. 
The unique solvability of the main equation can be proved by using the standard technique of the method of spectral mappings.

Solving the main equation \eqref{main-eq}, we obtain the sequence $\psi(x)$, which can be used for constructing $\{ \varphi_{n,i}(x)\}_{n \ge 1, \, i = 0, 1}$. The main result of this paper is the following theorem, which provides reconstruction formulas for $\sigma(x)$, $r_1(\la)$, and $r_2(\la)$. Put $\hat M(\la) = M(\la) - \tilde M(\la)$. Let us consider the contour $\Gamma_N := \Bigr\{\la \in \mathbb C \colon |\la| = \Bigr( N + \frac{1}{2} \Bigl)^2\Bigl\}$, where $N$ is such an integer that $\la_{ni} \in \mbox{int} \, \Gamma_N$ if $m_{ni} > 1$. So, this circle contains all the eigenvalues with multiplicity more than one.

\begin{thm} \label{mainthm}
	In the case $M_1 = M_2$, the solution of Inverse Problem~\ref{ip:main3} can be found by the reconstruction formulas
	\begin{align} \notag 
	\sigma(x) &= -\frac{1}{\pi i}\oint_{\Gamma_N} {\Bigg(\tilde \varphi(x, \mu)\varphi(x, \mu)-\frac{1}{2}\Bigg)\hat M(\mu)}d\mu \\ \label{sigma_series}
	& - 2\sum_{k=N+1}^\infty \sum_{j = 0}^1 (-1)^j\alpha_{kj}\Bigg(\tilde \varphi_{k,j}(x)\varphi_{k,j}(x) - \frac{1}{2}\Bigg),
	\end{align}
	\begin{align} \notag
	r_1(\la) = \prod_{k = 1}^{M_1} (\la - \la_{k0}) \prod_{k = M_1+1}^{\infty} & \frac{\la - \la_{k0}}{\la - \la_{k1}}\Bigg(1 - \frac{1}{2\pi i} \oint_{\Gamma_N} {\frac{\tilde\varphi'(\pi, \mu)\varphi(\pi, \mu)}{\la - \mu} M(\mu)}d\mu \\ \label{r1_series}
	& - \sum_{k=N+1}^\infty \frac{\alpha_{k0}\tilde\varphi'_{k,0}(\pi)\varphi_{k,0}(\pi)}{\la - \la_{k0}} \Bigg),
	\end{align}
	\begin{align}\notag 
	r_2(\la) = & \prod_{k = 1}^{M_1} (\la - \la_{k0}) \prod_{k = M_1+1}^{\infty} \frac{\la - \la_{k0}}{\la - \la_{k1}}\Bigg(\frac{1}{2 \pi i} \oint_{\Gamma_N} {\frac{\tilde \varphi' (\pi, \mu) \varphi^{[1]} (\pi, \mu)}{\la - \mu} M(\mu)}d\mu \\ \notag
	&  -\frac{1}{2\pi i} \oint_{\Gamma_N} (\tilde \varphi (\pi, \mu) \varphi (\pi, \mu) - 1) \hat M(\mu)d\mu + \sum_{k=N+1}^\infty \frac{\alpha_{k0}\tilde \varphi'_{k,0} (\pi) \varphi^{[1]}_{k,0} (\pi)}{\la-\la_{k0}} \\ \label{r2_series}
	& - \sum_{k = N+1}^{\infty} \sum_{j=0}^1 (-1)^j\alpha_{kj}(\tilde \varphi_{k,j} (\pi) \varphi_{k,j} (\pi) - 1)\Bigg),
	\end{align}
	where the series in \eqref{sigma_series} converges in $L_2(0,\pi)$ and the series and the infinite products in \eqref{r1_series} and \eqref{r2_series} converge absolutely and uniformly on compact sets in $\mathbb C \setminus \{ \la_{ni}\}_{n \ge 1, \, i = 0, 1}$.
\end{thm}

Notice that we can find the contour integral in \eqref{sigma_series} using the Residue theorem:
\begin{gather*}
\sigma(x) = - 2\sum_{i=0}^1 \sum_{k \in I_i} (-1)^{i} \left(\sum_{j=0}^{m_{ki} - 1} \alpha_{k+j,i} \sum_{p=0}^{j} \tilde \varphi_{k+p, i}(x)\varphi_{k+(j-p), i}(x) - \frac{1}{2} \alpha_{ki}\right),
\end{gather*}
where $I_i$ and $m_{ki}$ are defined in \eqref{not}.

Analogously, one can compute the contour integrals in \eqref{r1_series}--\eqref{r1_series}, but representations with contour integrals are more convenient for the purposes of our paper.

The derivation of the formula \eqref{sigma_series} for $\sigma(x)$ is based on the idea of substitution the series \eqref{series} into equation \eqref{eqv}. Then, we construct the Weyl solution of the problem $\mathcal L$ as a series analogous to \eqref{series} and use it to obtain the formulas \eqref{r1_series} and \eqref{r2_series} for the polynomials. This idea helps us to guess the form of the reconstruction formulas. However, it is inconvenient to use this idea for rigorous proofs, since the series do not converge absolutely and the potential is singular. Therefore, we prove Theorem~\ref{mainthm} by using the approximation approach. That is, we construct the approximations $(\sigma^K, r_1^K, r_2^K)$ to $(\sigma, r_1, r_2)$ by finite spectral data $\{ \la_n, \alpha_n \}_{n = 1}^K$ and pass to the limit as $K \to \infty$. These approximations can also be used for numerical solution of the inverse problem. 

Relying on the main equation \eqref{main-eq} and the reconstruction formulas of Theorem~\ref{mainthm}, we arrive at Algorithm~\ref{alg:sing} for solving Inverse Problem~\ref{ip:main}. Note that our reconstruction formulas are novel even for the case of regular potential $q \in L_2(0,\pi)$, which was considered by Freiling and Yurko~\cite{FrYu}. Therefore, in Section~\ref{sec:reg}, we apply our results to the regular case.

\section{Preliminaries and uniqueness} \label{sec:uniq}

In this section, we provide some preliminaries and prove the uniqueness theorem (Theorem~\ref{thm:uniq}).

Throughout this paper, we use the \textbf{notations}:

\begin{enumerate}
	\item $\la = \rho^2$, $\tau := \mbox{Im} \, \rho$, $\rho_n = \sqrt{\la_n}$, $\arg \rho_n \in \left[-\tfrac{\pi}{2}, \tfrac{\pi}{2} \right)$.
	\item The notation $\varkappa_a(\rho)$ is used for various entire functions in $\rho$ of exponential type not greater than a and such that $\int_{\mathbb{R}} |\varkappa_a(\rho)|^2 d\rho < \infty$. By the Paley-Wiener Theorem,
	\begin{gather*}
	\varkappa_a(\rho) = \int_{-a}^{a} f(x)e^{i\rho x} dx, \quad f \in L_2(-a, a).
	\end{gather*}
	Further we need the estimate $\varkappa_a(\rho) = o(e^{|\tau|a})$ as $|\rho| \rightarrow \infty$ uniformly by arg $\rho$.
	\item The notation $\{\varkappa_n\}$ is used for various sequences from $l_2$: $\sum_{n} |\varkappa_n|^2 < \infty$.
\end{enumerate}

Consider the boundary value problem $\mathcal L = \mathcal L(\sigma, r_1, r_2)$. Recall that $\varphi(x, \la)$ and $\psi(x, \la)$ are the solutions of equation \eqref{eqv} satisfying the initial conditions: $\varphi(0, \la) = 1$, $\varphi^{[1]}(0,\la) = 0$, $\psi(\pi, \la) = r_1(\la)$, $\psi^{[1]}(\pi, \la) = -r_2(\la)$. It can be easily seen that the eigenvalues of the problem $\mathcal L$ coincide with the zeros of the entire characteristic function
\begin{gather} \label{charfun}
\Delta(\la) = \psi(x, \la)\varphi^{[1]}(x, \la)-\varphi(x, \la)\psi^{[1]}(x, \la),
\end{gather}
which does not depend on the variable $x$. Substituting $x=\pi$ into \eqref{charfun}, we get the following representation:
\begin{gather} \label{charfun1}
\Delta(\la) = r_1(\la)\varphi^{[1]}(\pi, \la) + r_2(\la)\varphi(\pi, \la).
\end{gather}

In order to obtain the asymptotics of $\Delta(\la)$, we need the following proposition.

\begin{prop}[\cite{BondTamkang}] \label{prop1}
The functions $\varphi(x, \la)$ and $\varphi^{[1]}(x, \la)$ can be represented as follows:
\begin{gather*}
    \varphi(x, \la) = \cos\rho x + \int_0^{x}  \mathcal{K}_1(x, t) \cos\rho t \, dt, \\
    \varphi^{[1]}(x, \la) = -\rho\sin\rho x + \rho\int_0^{x} \mathcal{K}_2(x, t) \sin\rho t \, dt + \mathcal{C}(x),
\end{gather*}
where the functions $\mathcal{K}_j(x, t)$, $j = 1, 2$, belong to $\mathcal{D}_{0, \pi}$ and the function $\mathcal{C}(x)$ is continuous on $[0, \pi]$. Here and below, $\mathcal{D}_{0, \pi}$ is the class of functions $\mathcal K(x, t)$ that are square integrable in the region $\{(x, t): 0 < t < x < \pi\}$ and fulfill the conditions $\mathcal K(x, .) \in L_2(0,x)$ for each fixed $x \in (0,\pi]$ and $\sup_x \| \mathcal K(x, .) \|_{L_2(0,x)} < \infty$.
\end{prop}

Substituting the asymptotics of Proposition~\ref{prop1} into \eqref{charfun1}, we get
\begin{equation}\label{asympt_charfun}
\Delta(\la) =
    \left\{\begin{array}{ll}
        -\rho^{2M_1+1} \sin\rho\pi + \rho^{2M_1+1} \varkappa_\pi(\rho) + O(|\rho|^{2M_1}e^{\pi|\tau|}), \quad & M_1 \ge M_2, \\
        \rho^{2M_2} \cos\rho\pi + \rho^{2M_2} \varkappa_\pi(\rho) + O(|\rho|^{2M_2-1}e^{\pi|\tau|}), \quad & M_1 < M_2. \\
    \end{array}\right.
\end{equation}

Now we are ready to prove Lemma~\ref{eigen_asymp} on the eigenvalue asymptotics.

\begin{proof}[Proof of Lemma~\ref{eigen_asymp}]
For definiteness, suppose that $M_1 \ge M_2$. The opposite case can be studied similarly. According to \eqref{asympt_charfun}, the entire characteristic function $\Delta(\la)$ has the asymptotics
\begin{gather} \label{asymptDelta}
\Delta(\la) = -\rho^{2M_1+1} \sin\rho\pi + \rho^{2M_1+1} \varkappa_\pi(\rho) + O(|\rho|^{2M_1}e^{\pi|\tau|}).
\end{gather}

Using the standard approach based on Rouch\'e's Theorem (see, e.g., \cite[Theorem~1.1.3]{FY01}), one can easily show that the zeros of $\Delta(\la)$ are asymptotically close to the zeros of the following function:
$$
\Delta^0(\la) = -\rho^{2M+1} \sin \rho \pi = -\la^{M+1} \frac{\sin \sqrt{\la} \pi}{\sqrt \la},
$$
which are $\la_1^0 = \la_2^0 = \dots = \la_{M+1}^0 = 0$, $\la_n^0 = (n - M_1 - 1)^2$, $n > M_1 + 1$. Namely, we get
\begin{equation} \label{roughasympt}
\rho_n = n - M_1 - 1 + \varepsilon_n, \quad \varepsilon_n = o(1), \quad n \to \infty.
\end{equation}

In order to obtain a more precise estimate of the remainder term, we put $\rho = \rho_n$ with sufficiently large $n$ in \eqref{asymptDelta}:
\begin{equation} \label{sm1}
\rho_n^{-(2M_1+1)} \Delta(\rho_n^2) = -\sin \rho_n \pi + \varkappa_{\pi}(\rho_n) + O\left(n^{-1} \right) = 0.
\end{equation}

Using the Taylor series at zero, we get
\begin{align*}
\sin\rho_n \pi & = (-1)^m \sin (\varepsilon_n \pi) = 
(\varepsilon_n \pi + O(\varepsilon_n^3)), \quad m := n - M_1 - 1, \\
\varkappa_\pi(\rho_n) & =  \int_{-\pi}^{\pi} f(x) e^{i(n+\varepsilon_n)x} dx \\ & = \int_{-\pi}^{\pi} f(x)e^{inx} dx + i\varepsilon_n \int_{-\pi}^{\pi} xf(x)e^{inx} dx + O(\varepsilon_n^2) = \varkappa_n + o(\varepsilon_n).
\end{align*}

Substituting these representations into \eqref{sm1}, we can get that $\varepsilon_n = \varkappa_n$. Together with \eqref{roughasympt}, this concludes the proof.
\end{proof}

 Analogously, we obtain the following asymptotic formula for the weight numbers:
\begin{gather}\label{alpha_asymp}
\alpha_n = \dfrac{2}{\pi} + \varkappa_n, \quad n \ge 1.
\end{gather}

Using this asymptotics \eqref{asympt_charfun}, we can get the following inequality in the case $M_1 \ge M_2$ (see formula (16) in \cite{FrYu}):
\begin{gather}\label{abs_delta}
|\Delta(\la)| \ge C_\delta|\rho|^{2M_1+1}e^{\pi|\tau|}, \quad \rho \in G_\delta = \{\rho: |\rho-n| \ge \delta, n \in \mathbb{Z}\}.
\end{gather}

Let us introduce the Weyl solution of equation \eqref{eqv}:
\begin{gather*}
\Phi(x, \la) = -\frac{\psi(x, \la)}{\Delta(\la)},
\end{gather*}
which satisfies the boundary conditions
\begin{equation} \label{bcPhi}
\Phi^{[1]}(0, \la)=1, \quad 
r_1(\la)\Phi^{[1]}(\pi, \la)+r_2(\la)\Phi(\pi, \la)=0.
\end{equation}

In addition, denote by $S(x, \la)$ the solution of \eqref{eqv} satisfying the initial conditions $S(0, \la) = 0$, $S^{[1]}(0,\la) = 1$.
One can easily show that
\begin{gather} \label{Phi}
\Phi(x, \la) = S(x, \la)+M(\la)\varphi(x, \la),
\end{gather}
where $M(\la)$ is the Weyl function defined by \eqref{defM}.

Proceed to the proof of the uniqueness theorem.

\begin{proof}[Proof of Theorem~\ref{thm:uniq}]
Consider two boundary value problems $\mathcal L = \mathcal L(\sigma, r_1, r_2)$ and $\tilde {\mathcal L} = \mathcal L(\tilde \sigma, \tilde r_1, \tilde r_2)$ such that $(r_1, r_2) \in \mathcal N$, $(\tilde r_1, \tilde r_2) \in \mathcal N$, and $M(\la) \equiv \tilde M(\la)$.
Let us define the matrix of spectral mappings $P(x, \la)$ as follows:
\begin{equation*}
P(x, \la)
	\begin{pmatrix}
	\tilde\varphi(x, \la) & \tilde\Phi(x, \la)\\
	\tilde\varphi^{[1]}(x, \la) & \tilde\Phi^{[1]}(x, \la)
	\end{pmatrix}
	=
	\begin{pmatrix}
	\varphi(x, \la) & \Phi(x, \la)\\
	\varphi^{[1]}(x, \la) & \Phi^{[1]}(x, \la)
	\end{pmatrix}
\end{equation*}

So, we get:
\begin{equation} \label{phi_from_p}
	\begin{cases}
	P_{11}(x, \la)\tilde\varphi(x, \la)+P_{12}(x, \la)\tilde\varphi^{[1]}(x, \la)=\varphi(x, \la),\\
	P_{11}(x, \la)\tilde\Phi(x, \la)+P_{12}(x, \la)\tilde\Phi^{[1]}(x, \la)=\Phi(x, \la).
	\end{cases}
\end{equation}

The determinant of this system equals $1$, so
\begin{equation} \label{p11p12}
	\begin{cases}
	P_{11}(x, \la)=\tilde\Phi^{[1]}(x, \la)\varphi(x, \la)-\Phi(x, \la)\tilde\varphi^{[1]}(x, \la),\\
	P_{12}(x, \la)=\Phi(x, \la)\tilde\varphi(x, \la)-\tilde\Phi(x, \la)\varphi(x, \la).
	\end{cases}
\end{equation}

We can also get other representations of $P_{11}$ and $P_{12}$:
\begin{equation} \label{relP}
	\begin{cases}
	P_{11}(x, \la)=1+\varphi(x, \la)(\tilde\Phi^{[1]}(x, \la)-\Phi^{[1]}(x, \la))+\Phi(x, \la)(\varphi^{[1]}(x, \la)-\tilde\varphi^{[1]}(x, \la)),\\
	P_{12}(x, \la)=-\dfrac{\tilde\varphi(x, \la)\psi(x, \la)}{\Delta(\la)}+\dfrac{\varphi(x, \la)\tilde\psi(x, \la)}{\tilde\Delta(\la)}.
	\end{cases}
\end{equation}

Recall that $M(\la) = -\dfrac{\psi(0,\la)}{\psi^{[1]}(0,\la)}$. One can easily show that the functions $\psi(0,\la)$ and $\psi^{[1]}(0,\la)$ do not have common zeros (see Lemma~1 in \cite{FrYu}). Since $M(\la) \equiv \tilde M(\la)$, their poles coincide: $\la_n = \tilde \la_n$, $n \ge 1$ (counting with multiplicities). Taking the asymptotics \eqref{la_asymp} into account, we conclude that the both problems $\mathcal L$ and $\tilde{\mathcal L}$ have the same type (either $M_1 = M_2$ or $M_1 = M_2-1$). Moreover, $M_j = \tilde M_j$, $j = 1, 2$. 

For definiteness, consider the case $M_1 = M_2$. The other case is analogous.
Using Proposition~\ref{prop1}, we get:
\begin{gather*} 
\varphi(x, \la) = O(e^{x|\tau|}), \quad \varphi^{[1]}(x, \la) = O(|\rho|e^{x|\tau|}), \\ 
\varphi^{[1]}(x, \la) - \tilde\varphi^{[1]}(x, \la) = o(|\rho|e^{x|\tau|}), \quad |\rho| \rightarrow \infty.
\end{gather*}

Analogously, one can obtain the estimates
\begin{gather*}
\psi(x, \la) = O(|\rho|^{2M_1}e^{(\pi-x)|\tau|}), \quad \psi^{[1]}(x, \la) = O(|\rho|^{2M_1 + 1}e^{(\pi-x)|\tau|}), \\
\psi^{[1]}(x, \la) - \tilde \psi^{[1]}(x, \la) = o(|\rho|^{2M_1 + 1}e^{(\pi-x)|\tau|}), \quad |\rho| \to \infty.
\end{gather*}

Using these estimates together with \eqref{abs_delta} and \eqref{relP}, we obtain
\begin{gather} \label{p_as}
P_{11}(x, \la) = 1 + o(1), \quad P_{12}(x, \la) = o(1), \quad \rho \in G_\delta, \: |\rho| \to \infty.
\end{gather}

According to \eqref{Phi}, \eqref{p11p12}, and $M(\la) = \tilde M(\la)$ we get:
\begin{gather*}
P_{11} = \varphi(x, \la) \tilde S^{[1]}(x, \la) - S(x, \la) \tilde \varphi^{[1]}(x, \la), \\
P_{12} = S(x, \la) \tilde \varphi(x, \la) - \varphi(x, \la) \tilde S(x, \la),
\end{gather*}
so $P_{11}(x, \la)$ and $P_{12}(x, \la)$ are entire in $\la$, then by Liouville's Theorem $P_{11}(x, \la) \equiv 1$, $P_{12}(x, \la) \equiv 0$. Hence $\varphi(x, \la) \equiv \tilde\varphi(x, \la)$, $\Phi(x, \la) \equiv \tilde\Phi(x, \la)$.

Substituting $\varphi(x, \la)$ into $\mathit{l}y=\la y$ and $\tilde{\mathit{l}}y=\la y$, we get
\begin{gather*}
((\sigma(x)-\tilde\sigma(x))\varphi(x, \la))' = (\sigma(x)-\tilde\sigma(x))\varphi'(x, \la).
\end{gather*}
The function $(\sigma(x)-\tilde\sigma(x))\varphi(x, \la)$ is absolutely continuous on $[0, \pi]$. Fix such $\la$ that $\varphi(x, \la) \neq 0$, $x \in [0, \pi]$. Then, the function $(\sigma(x)-\tilde\sigma(x))$ is absolutely continuous on $[0, \pi]$ and $(\sigma(x)-\tilde\sigma(x))' = 0$. So $(\sigma(x)-\tilde\sigma(x))$ is a constant.
As $\varphi^{[1]}(0, \la) = \tilde\varphi^{[1]}(0, \la) = 0$, so we get
$\sigma(0)-\tilde\sigma(0) = 0$.
Since $\sigma(x)-\tilde\sigma(x)$ is a constant, then $\sigma(x) \equiv \tilde\sigma(x)$.

As $\Phi(x, \la) = \tilde \Phi(x, \la)$, then from right boundary condition \eqref{bcPhi} for the Weyl solution we get:
\begin{gather} \label{difr}
r_1(\la)\tilde r_2(\la) - r_2(\la)\tilde r_1(\la) = 0.
\end{gather}

Recall that $r_1(\la)$ and $r_2(\la)$ as well as $\tilde r_1(\la)$ and $\tilde r_2(\la)$ are relatively prime polynomials and $(r_1, r_2), (\tilde r_1, \tilde r_2) \in \mathcal N$. Therefore, \eqref{difr} implies $r_1(\la) \equiv \tilde r_1(\la)$ and $r_2(\la) \equiv \tilde r_2(\la)$.
\end{proof}

\section{Main equation} \label{sec:maineq}

In this section, we obtain the main equation of Inverse Problem~\ref{ip:main3} in the Banach space $m$ of bounded infinite sequences. First, we choose the model problem. Second, we integrate the entries $P_{11}(x, \la)$ and $P_{12}(x, \la)$ of the matrix of spectral mappings by special contours in the $\la$-plane and derive the infinite linear system \eqref{phi_ni}. Finally, we transform \eqref{phi_ni} to a linear equation in $m$.

Starting from this section, we agree that $M_1 = M_2$. The case $M_1 = M_2 - 1$ can be considered analogously. 

Let us introduce the model problem $\tilde {\mathcal{L}} = \mathcal{L}(\tilde \sigma, \tilde r_1, \tilde r_2)$:
\begin{gather*}
    \tilde {\mathit{l}}y(x) = \la y(x), \quad x \in (0, \pi), \\
    y^{[1]}(0) = 0, \quad \tilde r_1(\la)y^{[1]}(\pi) + \tilde r_2(\la)y(\pi) = 0,
\end{gather*}
where $M_1 = \tilde M_1$. For example, we can take $\tilde \sigma(x) \equiv 0$, $\tilde r_1(\la) = \la^{M_1}$ and $\tilde r_2(\la) \equiv 0$. So, we get the following problem $\tilde {\mathcal{L}}$:
\begin{gather} \label{eqv_tilde}
    -y''(x) = \la y(x), \quad x \in (0, \pi), \quad
    y'(0) = 0, \quad \la^{M_1}y'(\pi) = 0.
\end{gather}

We agree that, if a certain symbol $\gamma$ denotes an object related to $\mathcal{L}$, then the symbol $\tilde \gamma$ with tilde will denote the analogous object related to $\tilde {\mathcal{L}}$ and $\hat \gamma = \gamma - \tilde \gamma$.

It is easy to show that problem \eqref{eqv_tilde} has the following spectral data:
\begin{gather}\label{tildesd}
\tilde \la_n =
    \left\{\begin{array}{ll}
        (n - M_1 - 1)^2, \quad & n > M_1,\\
        0, \quad &  1 \le n \le M_1,
            \end{array}\right. \quad
\tilde \alpha_n =
	\left\{\begin{array}{ll}
        \frac{2}{\pi}, \quad & n > M_1 + 1,\\
        0, \quad &  2 \le n \le M_1 + 1,\\
        \frac{1}{\pi}, \quad &  n = 1.
            \end{array}\right.
\end{gather}

Furthermore, the equation $\tilde l y = \la y$ has the solutions $\tilde \varphi(x, \la) = \cos \rho x$ and $\tilde \psi(x,  \la) = \la^{M_1} \cos \rho (\pi - x)$.
Hence, the Weyl function and the Weyl solution of the problem $\tilde{\mathcal L}$ have the form
\begin{gather} \label{tildeMmain}
\tilde M(\la) = \dfrac{\cos \rho \pi}{\rho \sin \rho \pi}, \quad
\tilde \Phi(x, \la) = \dfrac{\cos \rho (\pi - x)}{\rho \sin \rho \pi},
\end{gather}
respectively. Note that $\tilde M(\la)$ coincides with the Weyl function of the problem
$$
-y''(x) = \la y(x), \quad x \in (0, \pi), \quad y'(0) = y'(\pi) = 0.
$$ 

However, for convenience, we add $M_1$ zero eigenvalues and consider the model problem $\tilde{\mathcal L}$ of form \eqref{eqv_tilde}.

Introduce the notation $\langle y, z \rangle := y z' - y' z$.
Define the auxiliary functions
\begin{gather} \label{funcA}
\tilde D(x, \la, \mu) = \frac {\langle \tilde \varphi(x, \la), \tilde \varphi(x, \mu)\rangle}{\la - \mu} = \int_0^x {\tilde \varphi(t, \la) \tilde \varphi(t, \mu)}dt, \quad \tilde A(x, \la, \mu) = \tilde D(x, \la, \mu) \hat M(\mu).
\end{gather}
Also, recall the notations \eqref{not}.

Similarly to Lemma~3 in \cite{FrYu}, we obtain the following result.

\begin{lem} \label{A}
The following representation holds:
\begin{gather*}
\tilde A (x, \la, \mu) = \sum_{j=0}^{m_{ki}-1}\frac{\tilde A_{k+j, i}(x, \la)}{(\mu - \la_{ki})^{j+1}}+\tilde A_{k, i}^0(x, \la, \mu),
\end{gather*}
where $\tilde A_{k+j,i} (x, \la) = \mathop{\mathrm{Res}}\limits_{\mu = \la_{ki}} \tilde A(x, \la, \mu)(\mu - \la_{ki})^j$, are the coefficients of the principal part of $\tilde A (x, \la, \mu)$ in a neighbourhood of the point $\mu = \la_{kj}$ and $\tilde A_{k, i}^0(x, \la, \mu)$ is analytic with respect to $\mu$ near $\la_{kj}$.
\end{lem}

Denote
\begin{gather*}
\tilde Q_{n+p, i; k, j} (x) = \frac{1}{p!} \frac{\partial^p}{\partial \la^p} \tilde A_{k, j}(x, \la)\bigg|_{\la=\la_{n, i}}, \quad n \in I_i, \quad p=\overline{0, m_{ni}-1}, \quad i=0, 1.
\end{gather*}

\begin{lem}
The following relations hold for $x \in [0,\pi]$:
\begin{gather} \label{varphifromtilde}
\tilde \varphi (x, \la) = \varphi (x, \la) + \sum_{i=0}^1 (-1)^i \sum_{k \in I_i} \sum_{j=0}^{m_{k_i}-1} \tilde A_{k+j, i}(x, \la) \varphi_{k+j, i}(x), \quad n \ge 1,
\\ \label{phi_ni}
\tilde \varphi_{n, i}(x) = \varphi_{n, i}(x) + \sum_{k=1}^{\infty} (\tilde Q_{n, i; k, 0} (x) \varphi_{k, 0}(x) - \tilde Q_{n, i; k, 1} (x) \varphi_{k, 1}(x)).
\end{gather}

The series in \eqref{varphifromtilde} converges in the sense
\begin{equation} \label{conv}
\lim_{K \to \infty} \sum_{i = 0}^1 (-1)^i \sum_{k \in I_i, \, k \le K} (\dots)
\end{equation}
uniformly with respect to $\la$ on compact sets of $\mathbb C$.
\end{lem}

\begin{proof}		
Define the regions
\begin{gather*}
\Xi = \{\la = \rho^2: 0 < \mbox{Im}\, \rho < \tau_1\}, \quad c_N = \Bigr\{\la: |\la| < \Bigr( N + \frac{1}{2} \Bigl)^2\Bigl\},  \quad N \in \mathbb N,
\end{gather*}
where $\tau_1$ is chosen so that $\la_{ni} \in \Xi$ for all $n \ge 1$, $i = 0, 1$. Denote by $\Upsilon^+_N$, $\Upsilon^-_N$, and $\Gamma_N$ the borders of the regions $(\Xi \cap c_N)$,  $(c_N \setminus \Xi)$, and $c_N$, respectively (see Fig.~\ref{img:cont}). The contours $\Upsilon^+_N$, $\Upsilon^-_N$, and $\Gamma_N$ are supposed to have the counter-clockwise circuit.

\begin{figure}[h!]
	\centering
	\begin{tikzpicture}[scale = 0.5]
	\draw[dashed] (-5, 0) edge (5, 0);
	\draw[dashed] (0, -5) edge (0, 5);
	\draw (0, 0) circle (4);
	\draw (5, 4) .. controls (-5, 1) and (-5,-1) .. (5, -4);
	\draw  (-1,3.2) node{$\Upsilon_N^-$};
	\draw (1,1.9) node{$\Upsilon_N^+$};
	\draw (4.7,0.5) node{$\Gamma_N$};
	\filldraw (-1, 0) circle(2pt);
	\filldraw (1, 0) circle(2pt);
	\filldraw (3, 0) circle(2pt);
	\end{tikzpicture}
\caption{Contours}
\label{img:cont}
\end{figure}

Note that the functions $P_{11}(x, \la)$ and $P_{12}(x, \la)$ given by \eqref{p11p12} have the poles $\{ \la_{ni}\}_{n \ge 1, \, i = 0, 1}$ for each fixed $x \in [0, \pi]$.
Due to the Cauchy integral formula, we get for $\la \not\in \Xi$ and for a sufficienty large integer $N$ that
\begin{gather*}
P_{11}(x, \la) - 1 = -\frac{1}{2\pi i} \oint_{\Upsilon^-_N} {\frac{P_{11}(x, \mu) - 1}{\la - \mu}d\mu} = \frac{1}{2\pi i} \oint_{\Upsilon^+_N} {\frac{P_{11}(x, \mu)}{\la - \mu}d\mu} -\frac{1}{2\pi i} \oint_{\Gamma_N} {\frac{P_{11}(x, \mu) - 1}{\la - \mu}d\mu}, \\
P_{12}(x, \la) = -\frac{1}{2\pi i} \oint_{\Upsilon^-_N} {\frac{P_{12}(x, \mu)}{\la - \mu}d\mu} = \frac{1}{2\pi i} \oint_{\Upsilon^+_N} {\frac{P_{12}(x, \mu)}{\la - \mu}d\mu} -\frac{1}{2\pi i} \oint_{\Gamma_N} {\frac{P_{12}(x, \mu)}{\la - \mu}d\mu}.
\end{gather*}

Then, using \eqref{phi_from_p}), we obtain
\begin{gather} \label{int_varphi}
\varphi(x, \la) = \tilde \varphi(x, \la) + \frac{1}{2\pi i} \oint_{\Upsilon^+_N} { \frac{P_{11}(x, \mu) \tilde \varphi(x, \la) + P_{12}(x, \mu) \tilde \varphi^{[1]}(x, \la)}{\la - \mu}d\mu} + \varepsilon_N(x, \la),
\end{gather}
where
\begin{gather*}
\varepsilon_N(x, \la) = - \frac{1}{2\pi i} \oint_{\Gamma_N} { \frac {(P_{11}(x, \mu) - 1)\tilde \varphi(x, \la) + P_{12}(x, \mu) \tilde \varphi^{[1]}(x, \la)}{\la - \mu}d\mu}.
\end{gather*}

Due to \eqref{p_as}, we have $P_{11}(x, \la) = 1 + o(1)$ and $P_{12}(x, \la) = o(1)$, so $\varepsilon_N(x, \la) \rightarrow 0$, $N \rightarrow \infty$.

Substituting \eqref{Phi} and \eqref{p11p12} into \eqref{int_varphi}, we deduce
\begin{gather} \label{int_varphi3}
\varphi(x, \la) = \tilde \varphi(x, \la) + \frac{1}{2\pi i} \oint_{\Upsilon^+_N} { \tilde \varphi(x, \mu)  \tilde A (x, \la, \mu) d\mu} + \varepsilon_N(x, \la),
\end{gather}

Let us represent $\tilde A(x, \la, \mu)$ in \eqref{int_varphi3} by Lemma~\ref{A}, calculate the integral in \eqref{int_varphi3} by the Residue theorem, and pass to the limit as $N \to \infty$. Then, we arrive at \eqref{varphifromtilde} for $\la \not\in \Xi$. By using the spectral data asymptotics \eqref{la_asymp} and \eqref{alpha_asymp}, one can show that the series in \eqref{varphifromtilde} converges in the sense \eqref{conv} uniformly with respect to $\la$ on compact sets. By analytical continuation, we conclude that the relation \eqref{varphifromtilde} holds for all $\la \in \mathbb C$. Putting $\la = \la_{ni}$ and previously differentiating the both sides of \eqref{varphifromtilde} for multiple eigenvalues, we obtain \eqref{phi_ni}.
\end{proof}

In the same way, we can get the formula
\begin{gather} \label{Phifromtilde}
\tilde \Phi (x, \la) = \Phi (x, \la) + \sum_{i=0}^1 (-1)^i \sum_{k \in I_i} \sum_{j=0}^{m_{k_i}-1} \tilde B_{k+j, i}(x, \la) \varphi_{k+j, i}(x),
\end{gather}
where $\tilde B(x, \la, \mu) = \dfrac {\langle \tilde \Phi(x, \la), \tilde \varphi(x, \mu)\rangle}{\la - \mu} \hat M(\mu)$ and the series converges in the sense \eqref{conv} uniformly with respect to $\la$ on compact sets in $\mathbb C \setminus \{ \la_{ni}\}_{n \ge 1, \, i = 0, 1}$.

Next, proceed to the derivation of the main equation.
Denote
\begin{gather*}
\xi_n = |\rho_n - \tilde \rho_n| + |\alpha_n - \tilde \alpha_n|, \quad
\chi_{n} = \left\{\begin{array}{ll}
\xi_n^{-1}, \quad & \xi_n \neq 0,\\
0, \quad & \xi_n = 0,
\end{array}\right. \quad n \ge 1.
\end{gather*}

By virtue of the spectral data asymptotics \eqref{la_asymp} and \eqref{alpha_asymp}, we have $\{ \xi_n \} \in l_2$.

Let $V$ be the set of indexes $u = (n, i)$, where $n \ge 1$, $i=0, 1$. Define the vector $\psi(x) = (\psi_v(x))_{v \in V}$, where
\begin{gather} \label{psimain}
	\begin{pmatrix}
	\psi_{n0}(x)\\
	\psi_{n1}(x)
	\end{pmatrix} =
	\begin{pmatrix}
	\chi_{n} & -\chi_{n}\\
	0 & 1
	\end{pmatrix}
	\begin{pmatrix}
	\varphi_{n,0}(x)\\
	\varphi_{n,1}(x)
	\end{pmatrix}.
\end{gather}

Analogously, define $\tilde \psi(x)$ replacing $\varphi_{n,i}(x)$ by $\tilde \varphi_{n,i}(x)$. Furthermore, define the infinite matrix $\tilde H(x) = (\tilde H_{u,v}(x))_{u,v \in V}$:
\begin{gather} \label{Hmatr}
	\begin{pmatrix}
	\tilde H_{n0, k0}(x) & \tilde H_{n0, k1}(x)\\
	\tilde H_{n1, k0}(x) & \tilde H_{n1, k1}(x)
	\end{pmatrix} =
	\begin{pmatrix}
	\chi_{n} & - \chi_{n}\\
	0 & 1
	\end{pmatrix}
	\begin{pmatrix}
	\tilde Q_{n,0; k,0}(x) & \tilde Q_{n,0; k,1}(x)\\
	\tilde Q_{n,1; k,0}(x) & \tilde Q_{n,1; k,1}(x)
	\end{pmatrix}
	\begin{pmatrix}
	\xi_{k} & 1\\
	0 & -1
	\end{pmatrix}.
\end{gather}

Using Proposition~\ref{prop1}, \eqref{la_asymp}, \eqref{asymptDelta} and \eqref{alpha_asymp}, we can get the estimates
\begin{gather} \label{uneqs}
|\psi_{ni}(x)| \le C, \quad |\tilde \psi_{ni}(x)| \le C, \quad
|\tilde H_{ni, kj}(x)| \le \dfrac{C \xi_k}{|n-k|+1}, \quad (n,i), \, (k,j) \in V,
\end{gather}
where $C$ is a positive constant independent of $n$, $i$, $k$, $j$, and $x$.

Let us consider the Banach space $m$ of bounded infinite sequences $a = (a_v)_{v \in V}$ with the norm $\| a \|_m = \sup_{v} |a_v|$.
It follows from \eqref{uneqs} that, for each fixed $x \in [0, \pi]$,
$\psi(x)$ and $\tilde \psi(x)$ belong to $m$, $\tilde H(x)$ is a linear bounded operator from $m$ to $m$ and
\begin{gather*}
\|\tilde H(x)\|_{m \to m} \le C \sup_n \sum_k \dfrac{\xi_k}{|n - k| + 1} < \infty.
\end{gather*}

Recall that $E$ is the identity operator in $m$.

\begin{thm}\label{main_equation_thm}
For each fixed $x \in [0, \pi]$, vector $\psi(x) \in m$ satisfies the equation
\begin{gather}\label{main_equation}
\tilde \psi(x) = (E + \tilde H(x)) \psi(x)
\end{gather}
in the Banach space $m$. Equation \eqref{main_equation} is called $\bold {the\ main}$ $\bold {equation}$ of the inverse problem.
\end{thm}

\begin{proof}
By construction, we have the following relations:
\begin{gather*}
\tilde \psi_{n0} = \chi_n\left(\varphi_{n,0} - \varphi_{n,1} + \sum_k (\tilde Q_{n,0; k,0} \varphi_{k,0} - \tilde Q_{n,0; k,1} \varphi_{k,1} - \tilde Q_{n,1; k,0} \varphi_{k,0} + \tilde Q_{n,1; k,1} \varphi_{k,1})\right), \\
\tilde \psi_{n1} = \varphi_{n,1} + \sum_k (\tilde Q_{n,1; k,0} \varphi_{k,0} - \tilde Q_{n,1; k,1} \varphi_{k,1}).
\end{gather*}

It is not difficult to notice that
\begin{gather*}
\tilde \psi_{ni} = \psi_{ni} +  \sum_{k, j} \tilde H_{ni, kj}(x)\psi_{kj}, \quad (n, i), (k, j) \in V,
\end{gather*}
which is equivalent to \eqref{main_equation}.
\end{proof}

\begin{thm}\label{main_equation_thm_solvability}
For each fixed $x \in [0, \pi]$, the operator $(E + \tilde H(x))$ has a bounded inverse operator. Therefore, equation \eqref{main_equation} is uniquely solvable.
\end{thm}

Theorem~\ref{main_equation_thm_solvability} is proved analogously to Theorem~2 in \cite{FrYu}, so we omit the proof.

\section{Reconstruction formulas} \label{sec:solution}

In this section, we prove Theorem~\ref{mainthm} on the reconstruction formulas for the coefficients $\sigma(x)$, $r_1(\la)$, and $r_2(\la)$ of the problem $\mathcal L$. First, we prove the convergence of the series and the products in Theorem~\ref{mainthm}. Then, we have to show that the right-hand sides of the formulas \eqref{sigma_series}--\eqref{r2_series} equal to the coefficients $\sigma(x)$, $r_1(\la)$, and $r_2(\la)$ of $\mathcal L$. For this purpose, we use the approximation approach. We consider the finite spectral data $\{ \la_n, \alpha_n \}_{n = 1}^K$ and the corresponding main equation for sufficiently large $K$. It is proved that the main equation is uniquely solvable. Furthermore, we construct the solutions $\varphi^K(x, \la)$ and $\Phi^K(x, \la)$ by the finite data $\{ \la_n, \alpha_n \}_{n = 1}^K$, study their properties, and then build the functions $\sigma^K(x)$, $r_1^K(\la)$, and $r_2^K(\la)$. We prove their convergence to $\sigma(x)$, $r_1(\la)$, and $r_2(\la)$ in the corresponding spaces. Finally, we arrive at the validity of the formulas \eqref{sigma_series}--\eqref{r2_series}.

The proof of Theorem~\ref{mainthm} consists of Lemmas~\ref{lem:sigma}--\ref{spectraldatacon}.

\begin{lem} \label{lem:sigma}
The series \eqref{sigma_series} converges in $L_2(0, \pi)$.
\end{lem}

\begin{proof}
Consider the convergence of the series
\begin{gather}\label{serlem}
\mathscr S(x) :=  \sum_{k=N+1}^\infty \sum_{j=0}^1 (-1)^j\alpha_{kj}\Bigg(\tilde \varphi_{k,j}(x)\varphi_{k,j}(x) - \frac{1}{2}\Bigg),
\end{gather}
where $N$ is so large that all the multiple eigenvalues lie inside the circle $\Gamma_N$.

In this proof, we use the notation $\{ \theta_k(x) \}$ for various sequences such that the series $\sum_k |\theta_k(x)|$ converges uniformly with respect to $x \in [0,\pi]$, and $\delta_x(k) = \int_0^x {\mathcal{K}(x, t)\cos kt}dt$, where $\mathcal{K}(x, t) \in \mathcal{D}_{0, \pi}$ (the class $\mathcal{D}_{0, \pi}$ is defined in Proposition~\ref{prop1}).

From Proposition~\ref{prop1}, \eqref{la_asymp}, \eqref{tildesd}, and the relation $\tilde \varphi(x,\la) = \cos \rho x$, we get
\begin{gather} \label{varphik}
\varphi_{k,0}(x) = \cos kx - \varkappa_kx\sin kx + \delta_x(k+\varkappa_k) + \theta_k(x), \quad \varphi_{k,1}(x) = \cos kx + \delta_x(k), \\ \label{tildevarphik}
\tilde \varphi_{k,0}(x) = \cos kx - \varkappa_kx\sin kx + \theta_k(x), \quad \tilde \varphi_{k,1}(x) = \cos kx.
\end{gather}

Recall that $\alpha_{k,1} = \frac{2}{\pi}$ for $k > N$. Put $\varepsilon_n := \alpha_{n,0} - \alpha_{n,1}$. It follows from \eqref{alpha_asymp} that $\{ \varepsilon_n\} \in l_2$.
Thus, substituting \eqref{varphik} and \eqref{tildevarphik} into \eqref{serlem} and simplifying the result, we obtain
\begin{gather*}
\mathscr S(x) = \sum_{k=N+1}^\infty \left(-\frac{2}{\pi}\varkappa_kx\sin2kx + \frac{1}{2}\varepsilon_k\cos2kx + \frac{2}{\pi}(\delta_x(k+\varkappa_k) - \delta_x(k)) + \theta_k(x)\right).
\end{gather*}

 As $\{\varepsilon_k\}, \{\varkappa_k\} \in l_2$ and 
 $$
 \delta_x(k+\varkappa_k) - \delta_x(k) = -\int_0^x {\mathcal{K}(x, t) \varkappa_kt\sin kt} dt + \theta_k(x),
 $$ 
 we get that the series $\mathscr S(x)$ converges in $L_2(0, \pi)$.
\end{proof}

\begin{lem} \label{lem:conv_of_r}
The series and the products in \eqref{r1_series} and \eqref{r2_series} converge absolutely and uniformly on compact sets $|\la| \le C$, $|\la - \la_{ki}| \ge \delta$, $C, \delta > 0$.
\end{lem}

\begin{proof}
At first consider the products. As a finite product does not affect the convergence, then we consider only the product
\begin{gather*}
\mathscr P(x) := \prod_{k = M_1+1}^{\infty} \frac{\la - \la_{k0}}{\la - \la_{k1}} = \prod_{k = M_1 + 1}^{\infty} \left( 1 + \frac{\la_{k1} - \la_{k0}}{\la - \la_{k1}} \right).
\end{gather*}

From \eqref{la_asymp} and \eqref{tildesd} we get $\la_{k0} = (k + \varkappa_k)^2$, $\la_{k1} = k^2$. Hence
$$
\frac{\la_{k1} - \la_{k0}}{\la - \la_{k1}} \: \sim \: \frac{\varkappa_k}{k}.
$$

Therefore, the product $\mathscr P(x)$ converges absolutely and uniformly on compact sets without the eigenvalues.

Next, consider the series in \eqref{r2_series}:
\begin{align*} 
\mathscr S_1(\la) & := \sum_{k=N+1}^\infty \frac{\alpha_{k0}\tilde \varphi'_{k,0} (\pi) \varphi^{[1]}_{k,0} (\pi)}{\la-\la_{k0}}, \\
\mathscr S_2 & := \sum_{k=N+1}^\infty \sum_{j=0}^1 (-1)^j\alpha_{kj}
(\tilde \varphi_{k,j} (\pi) \varphi_{k,j} (\pi) - 1).
\end{align*}

The proof for the series in \eqref{r1_series} is analogous and even easier.
Using Proposition~\ref{prop1} and \eqref{la_asymp}, we get
$$
\varphi_{k,0}^{[1]}(\pi) = k \varkappa_{k,0}, \quad
\tilde \varphi_{k,0}'(\pi) = k \varkappa_{k,1}, \quad \{ \varkappa_{k,0}\}, \, \{ \varkappa_{k,1}\} \in l_2.
$$

If $|\la|$ is bounded and $k$ is sufficiently large, then $|\la - \la_{k0}| \ge c k^2$, $c > 0$. Hence
$$
|\mathscr S_1(\la)| \le \sum_{k = N+1}^{\infty} C |\varkappa_{k,0}| |\varkappa_{k,1}| < \infty.
$$

Thus, the series $\mathscr S_1(\la)$ converges absolutely and uniformly on compact sets.

It follows from \eqref{varphik} and \eqref{tildevarphik} that
$$
\varphi_{k,0}(\pi) \tilde \varphi_{k,0}(\pi) = 1 + \delta_{\pi}(k + \varkappa_k) + \theta_k, \quad 
\varphi_{k,1}(\pi) \tilde \varphi_{k,1}(\pi) = 1 + \delta_{\pi}(k) + \theta_k, 
$$
where $\delta_{\pi}(k) = \int_0^{\pi} f(t) \cos k t \, dt$, $f \in L_2(0,\pi)$, and $\{ \theta_k\}$ denotes various sequences of $l_1$. Furthermore, recall that $\alpha_{k0} = \frac{2}{\pi} + \varepsilon_k$, $\alpha_{k1} = \frac{2}{\pi}$, where $\{ \varepsilon_k\} \in l_2$. Hence
$$
\mathscr S_2 = \sum_{k = N+1}^{\infty} \left( \frac{2}{\pi} (\delta_{\pi}(k+\varkappa_k) - \delta_{\pi}(k)) + \varepsilon_k \delta_{\pi}(k+\varkappa_k) + \theta_k \right).
$$

Note that 
$$
\{ \delta_{\pi}(k) \} \in l_2, \quad 
\{ (\delta_{\pi}(k+\varkappa_k) - \delta_{\pi}(k)) \} \in l_1. 
$$

This implies the absolute convergence of the series $\mathscr S_2$.
\end{proof}

In order to prove Theorem~\ref{mainthm}, it remains to show that the series \eqref{sigma_series}, \eqref{r1_series}, and \eqref{r2_series} converge to the coefficients $\sigma(x)$, $r_1(\la)$, and $r_2(\la)$, respectively. For this purpose, we use the approximation approach.
Consider the values $\{\la_n^K, \alpha_n^K\}_{n \ge 1}$, where $K$ is a sufficiently large natural number:
\begin{gather} \label{sdK}
\la_n^K =
    \left\{\begin{array}{ll}
        \la_n, \quad & n \le K,\\
        \tilde \la_n, \quad & n > K,\\
    \end{array}\right.
    ,\quad
\alpha_n^K =
    \left\{\begin{array}{ll}
        \alpha_n, \quad & n \le K,\\
        \tilde \alpha_n, \quad & n > K.\\
    \end{array}\right.
\end{gather}

Similarly to $\tilde H(x)$, we introduce the operator $\tilde H^K(x) = (\tilde H^K_{u,v}(x))_{u,v \in V}$:
\begin{gather*}
\tilde H^K_{ni, kj}(x) = 
\left\{\begin{array}{ll}
        0, \quad & k > K,\\
        \tilde H_{ni, kj}(x), \quad & k \le K.\\
    \end{array}\right.
\end{gather*}

\begin{lem} \label{lem:solveHK}
For each fixed $x \in [0,\pi]$, the operator $(E + \tilde H^K(x))$ has a bounded inverse.
\end{lem}

\begin{proof} 
Let $f$ be an element of $m$. Then
\begin{gather*}
(\tilde H^K(x)f)_{ni} = \sum_{k=1}^{K} (\tilde H_{ni, k0}(x)f_{k0} + \tilde H_{ni, k1}(x)f_{k1}), \quad n \ge 1, i=0,1.
\end{gather*}

So we can estimate the norm 
\begin{gather*}
\|\tilde H(x) - \tilde H^K(x)\|_{\it m \to \it m} \le C \sup_n \sum_{k=K+1}^\infty \frac{\xi_k}{|n - k| + 1} \le C \sqrt{\sum_{k=K+1}^\infty {\xi^2_k}}.
\end{gather*}

As $\{ \xi_k \} \in l_2$, then $\|\tilde H(x) - \tilde H^K(x)\|_{\it m \to \it m} \to 0$ as $K \to \infty$. Together with Theorem~\ref{main_equation_thm_solvability}, this yields the claim.
\end{proof}

By virtue of Lemma~\ref{lem:solveHK}, the equation
\begin{gather*}
\tilde \psi(x) = (E + \tilde H^K(x)) \psi^K(x)
\end{gather*}
has a unique solution for each fixed $x \in [0,\pi]$. The first elements $\psi^K_{ni}(x)$, $n = \overline{1, K}$, $i=0,1$, of its solution $\psi^K(x)$ solve the finite system of linear equations
\begin{gather} \label{mainK}
\tilde \psi_{ni}(x) = \psi^K_{ni}(x) + \sum_{k=1}^{K} (\tilde H_{ni, k0}(x)\psi^K_{k0}(x) + \tilde H_{ni, k1}(x)\psi^K_{k1}(x)), \quad n = \overline{1, K}, i=0,1.
\end{gather}

Since all the functions in the system \eqref{mainK} are infinitely differentiable, then so does the solution. 
Define
\begin{gather} \label{defvarphiK}
\varphi^K_{n,0}(x) = \xi_n\psi^K_{n0}(x)+\psi^K_{n1}(x), \quad
\varphi^K_{n,1}(x) = \psi^K_{n1}(x), \quad n = \overline{1,K}.
\end{gather}

Using \eqref{mainK} and \eqref{defvarphiK}, we obtain the relation similar to \eqref{phi_ni}:
\begin{equation} \label{phiK_ni}
\tilde \varphi_{n,i}(x) = \varphi^K_{n,i}(x) + \sum_{k=1}^K (\tilde Q_{n,i; k,0}(x)\varphi^K_{k,0}(x) - \tilde Q_{n,i; k,1}(x)\varphi^K_{k,1}(x)), \quad n = \overline{1, K}.
\end{equation}

Construct the functions $\varphi^K(x, \la)$ and $\Phi^K(x, \la)$ by the following formulas, which are similar to \eqref{varphifromtilde} and \eqref{Phifromtilde}, respectively:
\begin{gather} \label{varphiKfromtilde}
\varphi^K(x, \la) = \tilde\varphi(x, \la) - \sum_{i=0}^1 \sum_{k \in I_i, \, k \le K} \sum_{j=0} ^{m_{ki}-1} \tilde A_{k+j, i}(x, \la) \varphi^K_{k+j, i}(x, \la), \\ \notag
\Phi^K(x, \la) = \tilde\Phi(x, \la) - \sum_{i=0}^1 \sum_{k \in I_i, \, k \le K} \sum_{j=0} ^{m_{ki}-1} \tilde B_{k+j, i}(x, \la) \varphi^K_{k+j, i}(x, \la).
\end{gather}

It follows from \eqref{phiK_ni} and \eqref{varphiKfromtilde} that
$\varphi^K_{n+j, i}(x) = \varphi^K_j(x, \la_{ni})$, where $\varphi^K_{j}(x, \la) = \frac{1}{j!} \frac{\partial^j}{\partial \la^j} \varphi^K(x, \la)$.

We can also get the integral representations
\begin{gather} \label{varphiKint}
\varphi^K(x, \la) = \tilde\varphi(x, \la) - \frac{1}{2\pi i} \oint_{\Gamma_K} {\tilde A(x, \la, \mu) \varphi^K(x, \mu)} d\mu, \\ \label{PhiKint}
\Phi^K(x, \la) = \tilde\Phi(x, \la) - \frac{1}{2\pi i} \oint_{\Gamma_K} {\tilde B(x, \la, \mu) \varphi^K(x, \mu)} d\mu.
\end{gather}

Construct the functions
\begin{align} \notag 
\sigma^K(x) &= -\frac{1}{\pi i}\oint_{\Gamma_N} {\Bigg(\tilde \varphi(x, \mu)\varphi^K(x, \mu)-\frac{1}{2}\Bigg)\hat M(\mu)}d\mu \\ \label{sigmaK_series}
& - 2 \sum_{k=N+1}^{K} \sum_{j=0}^1 (-1)^j\alpha_{kj}\Bigg(\tilde \varphi_{k,j}(x)\varphi^K_{k,j}(x) - \frac{1}{2}\Bigg),
\end{align}
\begin{align} \notag
r_1^K(\la) = \prod_{k = 1}^{M_1} (\la - \la_{k0}) \prod_{k = M_1+1}^{K} & \frac{\la - \la_{k0}}{\la - \la_{k1}}\Bigg(1 - \frac{1}{2\pi i} \oint_{\Gamma_N} {\frac{\tilde\varphi'(\pi, \mu)\varphi^K(\pi, \mu)}{\la - \mu} M(\mu)}d\mu \\ \label{r1K_series}
& - \sum_{k=N+1}^K \frac{\alpha_{k0}\tilde\varphi'_{k,0}(\pi)\varphi^K_{k,0}(\pi)}{\la - \la_{k0}} \Bigg),
\end{align}
\begin{align}\notag 
	r_2^K(\la) = & \prod_{k = 1}^{M_1} (\la - \la_{k0}) \prod_{k = M_1+1}^{K} \frac{\la - \la_{k0}}{\la - \la_{k1}}\Bigg(\frac{1}{2 \pi i} \oint_{\Gamma_N} {\frac{\tilde \varphi' (\pi, \mu) \varphi^{K [1]} (\pi, \mu)}{\la - \mu} M(\mu)}d\mu \\ \notag
	&  -\frac{1}{2\pi i} \oint_{\Gamma_N} (\tilde \varphi (\pi, \mu) \varphi^K (\pi, \mu) - 1) \hat M(\mu)d\mu + \sum_{k=N+1}^K \frac{\alpha_{k0}\tilde \varphi'_{k,0} (\pi) \varphi^{K [1]}_{k,0} (\pi)}{\la-\la_{k0}} \\ \label{r2K_series}
	& - \sum_{k = N+1}^{K} \sum_{j=0}^1 (-1)^j\alpha_{kj}(\tilde \varphi_{k,j} (\pi) \varphi_{k,j}^K (\pi) - 1)\Bigg).
	\end{align}

The quasi-derivative in \eqref{r2K_series} is defined as $y^{[1]} := y' - \sigma^K y$.

\begin{lem} \label{lem:degree}
$r_1^K(\la)$ is a polynomial of degree $M_1$ with the leading coefficient $1$, and $r_2^K(\la)$ is a polynomial of degree not greater than $M_1$.
\end{lem}
\begin{proof}
For simplicity, assume that the eigenvalues $\{ \la_{k0} \}_{k \ge 1}$ are simple and $\la_{k0} \ne \la_{n1}$. The general case requires technical changes. Then
\begin{align} \label{r1K_series_simple}
r_1^K(\la) = \prod_{k = 1}^{M_1} (\la - \la_{k0}) \prod_{k = M_1+1}^{K} \frac{\la - \la_{k0}}{\la - \la_{k1}}\Bigg(1 - \sum_{k=1}^K \frac{\alpha_{k0}\tilde\varphi'_{k,0}(\pi)\varphi^K_{k,0}(\pi)}{\la - \la_{k0}} \Bigg),
\end{align}

The expression in parentheses in \eqref{r1K_series_simple} can be represented as follows:
\begin{gather*}
\mathscr E(\la) :=  1 - \sum_{k=1}^K \frac{\alpha_{k0}\tilde\varphi'_{k,0}(\pi)\varphi^K_{k,0}(\pi)}{\la - \la_{k0}} = \frac{\prod\limits_{k = 1}^{K} (\la - \la_{k0}) - \sum\limits_{k=1}^{K-1} s_k\la^k}{\prod\limits_{k = 1}^{K} (\la - \la_{k0})} = \frac{\sum\limits_{k=1}^{K} s^*_k\la^k}{\prod\limits_{k = 1}^{K} (\la - \la_{k0})},
\end{gather*}
where $s_k, s^*_k \in \mathbb{C}$ and $s^*_K = 1$.
Then, we get
\begin{gather} \label{fracr1}
r_1^K(\la) = \frac{ S^*(\la)}{\prod\limits_{k = M_1+1}^{K} (\la - \la_{k1})}, \qquad S^*(\la) := \sum_{k=1}^{K} s^*_k\la^k.
\end{gather}

Let us show that $S^*(\la_{n1}) = 0$ for all $n = \overline{1, K}$.
From \eqref{varphiKint} we get:
\begin{gather*}
\varphi^K_{n,1}(\pi) = \tilde\varphi_{n,1}(\pi) - \frac{1}{2\pi i} \oint_{\Gamma_K} {\tilde A(\pi, \la_{n1}, \mu) \varphi^K(\pi, \mu)} d\mu.
\end{gather*}

Calculating the contour integral by the Residue theorem and taking the relation $\tilde \varphi_{n,1}(\pi) = 1$ into account, we obtain
\begin{gather*}
\varphi^K_{n,1}(\pi) = 1 - \sum_{k=1}^K \frac{\alpha_{k0}\tilde\varphi'_{k,0}(\pi)\varphi^K_{k,0}(\pi)}{\la_{n1} - \la_{k0}} +\varphi^K_{n,1}(\pi), \quad n = \overline{1, K},
\end{gather*}
and, consequently,
\begin{gather*}
\mathscr E(\la_{n1}) = 1 - \sum_{k=1}^K \frac{\alpha_{k0}\tilde\varphi'_{k,0}(\pi)\varphi^K_{k,0}(\pi)}{\la_{n1} - \la_{k0}} = 0 \quad n = \overline{1, K}.
\end{gather*}

Hence, $S^*(\la_{n1}) = 0$, $n = \overline{1,K}$, so one can reduce the fraction \eqref{fracr1} and obtain a polynomial.
Since the numerator degree is $K$ and the denominator degree is $(K - M_1)$, then $r_1^K(\la)$ is a polynomial of degree $M_1$.

The proof for $r_2^K(\la)$ is analogous. The only difference is that the leading coefficient of $r_2^K(\la)$ not necessarily equals $1$. 
\end{proof}

Consider the boundary value problem $\mathcal L^K := \mathcal L(\sigma^K, r_1^K, r_2^K)$. Clearly, the function $\sigma^K(x)$ is infinitely differentiable, so the problem $\mathcal L^K$ can be represented as follows:
$$
	-y'' + \dfrac{d}{dx} \sigma^K(x) y = \la y, \quad x \in (0, \pi),
$$
$$
	y'(0) - \sigma^K(0) y(0) = 0, \quad r_1^K(\la)(y'(\pi) - \sigma^K(\pi) y(\pi)) + r_2^K(\la) y(\pi) = 0.
$$

\begin{lem} \label{lem:sol}
The following relations hold:
\begin{gather*}
-\dfrac{d^2}{dx^2}\varphi^K(x, \la) + \frac{d}{dx} \sigma^K(x) \varphi^K(x, \la) = \la \varphi^K(x, \la), \quad x \in (0,\pi), \\ \varphi^K(0,\la) = 1, \quad \frac{d}{dx}\varphi^K(0,\la) = \sigma^K(0), \\
-\dfrac{d^2}{dx^2}\Phi^K(x, \la) + \frac{d}{dx} \sigma^K(x) \Phi^K(x, \la) = \la \Phi^K(x, \la), \quad x \in (0,\pi), \\
\frac{d}{dx}\Phi^K(0,\la) - \sigma^K(0) \Phi^K(0,\la) = 0, \\ r_1^K(\la)\left(\frac{d}{dx}\Phi^K(\pi,\la) - \sigma^K(\pi) \Phi^K(\pi,\la)\right) + r_2^K(\la) \Phi^K(\pi,\la) = 0.
\end{gather*}
Thus, $\Phi^K(x, \la)$ is the Weyl solution of the problem $\mathcal{L}^K$.
\end{lem}

Lemma~\ref{lem:sol} can be proved by direct calculations.

\begin{lem}
$\{\la^K_n, \alpha^K_n\}_{n \ge 1}$ are the spectral data for the boundary problem $\mathcal{L}^K$.
\end{lem}

\begin{proof}
From Lemma~\ref{lem:sol}, we get that $\Phi^K(x, \la)$ is the Weyl solution of the problem $\mathcal{L}^K$. So, we can find the Weyl function as $\Phi^K(0, \la)$. Then it follows from \eqref{PhiKint} that
\begin{gather*}
M^K(\la) = \tilde M(\la) + \frac{1}{2\pi i} \oint_{\Gamma_K} {\frac{M(\mu)}{\la - \mu}}d\mu - \frac{1}{2\pi i} \oint_{\Gamma_K} {\frac{\tilde M(\mu)}{\la - \mu}}d\mu.
\end{gather*}

Substituting \eqref{weyl_func} and counting the contour integrals using the Residue theorem, we get
\begin{gather*}
M^K(\la) = \sum_{k \in I} \sum_{j=0}^{m_{k} - 1} \frac{\alpha^K_{k+j}}{(\la-\la^K_k)^{j+1}},
\end{gather*}
Consequently, $\{\la^K_n, \alpha^K_n\}_{n \ge 1}$ are the spectral data for the boundary problem $\mathcal{L}^K$.
\end{proof}

Next, let us pass from the coefficients $(\sigma^K, r_1^K, r_2^K)$ to $(\sigma, r_1, r_2)$.

\begin{lem} \label{coeffscon}
$\sigma^K(x) \to \sigma(x)$ in $L_2(0, \pi)$, $r^K_1(\la) \to r_1(\la)$ and $r^K_2(\la) \to r_2(\la)$ uniformly on the compact sets as $K \to \infty$ (i.e. the corresponding coefficients of the polynomials converge), where $\sigma(x)$, $r_1(\la)$, and $r_2(\la)$ are defined by \eqref{sigma_series}, \eqref{r1_series}, and \eqref{r2_series}, respectively.
\end{lem}

\begin{proof}
The proof for $\sigma(x)$ generalizes the proof of Lemma~5.6 in \cite{BondTamkang} to the case of multiple eigenvalues. Therefore, we outline it briefly.

Let us introduce the following functions $\theta^K(x) = (\theta^K_{ni}(x))_{(n,i)\in V}$:
\begin{gather*}
\theta_{ni}^K(x) = \left\{\begin{array}{ll}
\varphi_{n,i}(x) - \varphi_{n,i}^K(x), \quad & n \le K, \\
0, \quad &  n > K,
\end{array}\right.
\end{gather*}
and $\zeta_K(x) = \{\zeta^K_{ni}(x)\}_{(n,i)\in V}$:
\begin{gather*}
\zeta_{ni}^K(x) = \left\{\begin{array}{ll}
-\sum_{k=K+1}^\infty (\tilde Q_{n,i; k,0}(x)\varphi_{k,0}(x) - \tilde Q_{n,i; k,1}(x)\varphi_{k,1}(x)), \quad & n \le K, \\
0, \quad &  n > K,
\end{array}\right.
\end{gather*}
	
According to the systems \eqref{phi_ni} and \eqref{phiK_ni}, we have for $n \le K$, $i=0, 1$:
\begin{gather*}
\tilde \varphi_{n,i}(x) = \varphi_{n,i}(x) + \sum_{k=1}^\infty (\tilde Q_{n,i; k,0}(x)\varphi_{k,0}(x) - \tilde Q_{n,i; k,1}(x)\varphi_{k,1}(x)), \\
\tilde \varphi_{n,i}(x) = \varphi^K_{n,i}(x) + \sum_{k=1}^K (\tilde Q_{n,i; k,0}(x)\varphi^K_{k,0}(x) - \tilde Q_{n,i; k,1}(x)\varphi^K_{k,1}(x)).
\end{gather*}

The subtraction yields
$$
\theta_{ni}^K(x) = -\sum_{k = 1}^{\infty} (\tilde Q_{n,i; k,0}(x)\theta^K_{k,0}(x) - \tilde Q_{n,i; k,1}(x)\theta^K_{k,1}(x)) + \zeta_{ni}^K(x).
$$

Using the latter relation in the same way as in \cite{BondTamkang}, we can get that 
\begin{equation} \label{thetaK}
\lim_{K \to \infty}\|\theta^K(x)\|_{l_2} = 0, \quad \lim_{K \to \infty} \| \{ \theta_{n0}^K(x) - \theta_{n1}^K(x) \} \|_{l_1} = 0, 
\end{equation}
uniformly by $x \in [0, \pi]$.

Let us also introduce the function $\theta^K(x, \la) := \varphi(x, \la) - \varphi^K(x, \la)$. Using \eqref{varphifromtilde} and \eqref{varphiKfromtilde}, we can get:
\begin{gather*}
\theta^K(x, \mu) = \sum_{i=0}^1 (-1)^i \sum_{k \in I_i, k \le K} \sum_{j=0}^{m_{k_i}-1} \tilde A_{k+j, i}(x, \la)(\varphi^K_{k+j, i}(x) - \varphi_{k+j, i}(x)) \\
 - \sum_{i=0}^1 (-1)^i \sum_{k \in I_i, k > K} \sum_{j=0}^{m_{k_i}-1} \tilde A_{k+j, i}(x, \la)\varphi_{k+j, i}(x).
\end{gather*}

It follows from \eqref{thetaK} and the convergence of the series in \eqref{varphifromtilde} that 
\begin{equation} \label{thetaK2}
\lim_{K \to \infty} \theta^K(x, \mu) = 0.
\end{equation}
uniformly by $x \in [0, \pi]$ and $\mu$ on the compact set $\Gamma_N$.

Subtracting \eqref{sigmaK_series} from \eqref{sigma_series}, we get
\begin{gather*}
\sigma(x) - \sigma^K(x) = -\frac{1}{\pi i} \oint_{\Gamma_N} {\tilde\varphi(x, \mu)(\varphi(x, \mu) - \varphi^K(x, \mu))\hat M(\mu)}d\mu - 2 \sum_{k=N+1}^{K} (\tilde\varphi_{k,0}(x)\theta^K_{k0}(x)\alpha_{k0} \\
- \tilde\varphi_{k,1}(x)\theta^K_{k1}(x)\alpha_{k1}) - 2\sum_{k=K+1}^\infty \Bigg(\alpha_{k0}\bigg(\tilde\varphi_{k,0}(x)\varphi_{k,0}(x) - \frac{1}{2}\bigg) - \alpha_{k1}\bigg(\tilde\varphi_{k,1}(x)\varphi_{k,1}(x) - \frac{1}{2}\bigg)\Bigg)
\end{gather*}

Since the series \eqref{sigma_series} converges by Lemma~\ref{lem:sigma}, then its remainder tends to zero in $L_2(0,\pi)$ and $K \to \infty$.
Taking \eqref{thetaK} and \eqref{thetaK2} into account, we conclude that 
$\sigma^K(x) \to \sigma(x)$ in $L_2(0, \pi)$.

Let us consider the difference
\begin{gather*}
r_1(\la) - r_1^K(\la) = \frac{\prod\limits_{k=1}^K(\la-\la_{k0})}{\prod\limits_{k=M_1+1}^K(\la-\la_{k1})}\Bigg( \prod\limits_{k=K+1}^\infty\frac{\la-\la_{k0}}{\la-\la_{k1}}\bigg(1 - \frac{1}{2\pi i} \oint_{\Gamma_N}{\frac{\tilde \varphi'(\pi, \mu)\varphi(\pi, \mu)}{\la - \mu} M(\mu)}d\mu \\ - \sum_{k=N+1}^{\infty}\frac{\alpha_{k0}\tilde\varphi'_{k,0}(\pi)\varphi_{k,0}(\pi)}{\la - \la_{k0}}\bigg) - 1 + \frac{1}{2\pi i} \oint_{\Gamma_N}{\frac{\tilde \varphi'(\pi, \mu)\varphi^K(\pi, \mu)}{\la - \mu} M(\mu)}d\mu - \sum_{k=N+1}^{K}\frac{\alpha_{k0}\tilde\varphi'_{k,0}(\pi)\varphi^K_{k,0}(\pi)}{\la - \la_{k0}}\Bigg).
\end{gather*}

As the infinite product converges (see Lemma~\ref{lem:conv_of_r}), then 
$$
\lim_{K \to \infty} \prod\limits_{k=K+1}^\infty\frac{\la-\la_{k0}}{\la-\la_{k1}} = 1. 
$$

Here and below, the convergence is uniform by $\la$ on compact sets excluding the eigenvalues.
Denote the finite product $\dfrac{\prod\limits_{k=1}^K(\la-\la_{k0})}{\prod\limits_{k=M_1+1}^K(\la-\la_{k1})}$ as $g(\la)$.

So, we can get:
\begin{gather*}
\lim_{K \to \infty}(r_1(\la) - r_1^K(\la)) = -g(\la)\lim_{K \to \infty}\Bigg( \frac{1}{2\pi i} \oint_{\Gamma_N}{\frac{\tilde \varphi'(\pi, \mu)\theta^K(\pi, \mu)}{\la - \mu}\hat M(\mu)}d\mu \\+ \sum_{k=N+1}^{K}\frac{\alpha_{k0}\tilde\varphi'_{k,0}(\pi)\theta^K_{k,0}(\pi)}{\la - \la_{k0}} + \sum_{k=K+1}^{\infty}\frac{\alpha_{k0}\tilde\varphi'_{k,0}(\pi)\varphi^K_{k,0}(\pi)}{\la - \la_{k0}}\Bigg).
\end{gather*}

Recall that the series $\sum\limits_{k=1}^{\infty}\dfrac{\alpha_{k0}\tilde\varphi'_{k,0}(\pi)\varphi^K_{k,0}(\pi)}{\la - \la_{k0}}$ converges by Lemma~\ref{lem:conv_of_r}.
Taking \eqref{thetaK} and \eqref{thetaK2} into account, we  conclude that
$r_1(\la) - r_1^K(\la) \to 0$, $K \to \infty$. 
The proof for $r_2^K(\la)$ is similar. Thus, we have proved the absolute and uniform convergence on compact sets in $\mathbb C \setminus \{ \la_{ni}\}_{n \ge 1,\, i = 0, 1}$. Since $r_1^K(\la)$ and $r_2^K(\la)$ are polynomials and so have no singularities, then the convergence is uniform on compact sets in the whole $\la$-plane.
\end{proof}

\begin{remark} \label{rem:poly}
Lemmas~\ref{lem:degree} and~\ref{coeffscon} together imply that the right-hand sides of \eqref{r1_series} and \eqref{r2_series} are polynomials of degree $M_1$ and not greater than $M_1$, respectively. Moreover, the polynomial given by \eqref{r1_series} has the leading coefficient $1$.
\end{remark}

Furthermore, one can prove the following lemma on the continuous dependence of the spectral data on $\sigma$, $r_1$, and $r_2$.

\begin{lem}\label{spectraldatacon}
Let $\sigma^K(x)$ and $\sigma(x)$ are any functions of $L_2(0,\pi)$ and let
$r_j^K(\la)$ and $r_j(\la)$ for $j = 1, 2$ be arbitrary polynomials (not necessarily given by the formulas \eqref{sigma_series}--\eqref{r2_series}, \eqref{sigmaK_series}--\eqref{r2K_series}) such that
$\sigma^K(x) \to \sigma(x)$ in $L_2(0, \pi)$, $r^K_1(\la) \to r_1(\la)$ and $r^K_2(\la) \to r_2(\la)$ uniformly on compact sets
as $K \to \infty$, $r_1^K(\la)$ have degree $M_1$ and the leading coefficient $1$ and $r_2^K(\la)$ have degree not greater than $M_1$. 
Then $\{\la^K_n, \alpha^K_n\} \to \{\la_n, \alpha_n\}$ as $K \to \infty$ in the following sense: for each fixed $n \ge 1$, $\la^K_n \to \la_n$ and, if the multiplicity of $\la^K_n$ coincides with the multiplicity of $\la_n$, then $\alpha^K_n \to \alpha_n$ as $K \to \infty$.
\end{lem}

\begin{proof}[Proof of Theorem~\ref{mainthm}]
By virtue of Lemmas~\ref{lem:sigma} and~\ref{lem:conv_of_r}, the right-hand sides of \eqref{sigma_series}, \eqref{r1_series}, and \eqref{r2_series} converge in the appropriate sense. Denote their limits by $\sigma^{\diamond}(x)$, $r_1^{\diamond}(\la)$, and $r_2^{\diamond}(\la)$, respectively. We have to prove that $\sigma^{\diamond} = \sigma$, $r_1^{\diamond} = r_1$, and $r_2^{\diamond} = r_2$. Consider the problem $\mathcal L^{\diamond} := \mathcal L(\sigma^{\diamond}, r_1^{\diamond}, r_2^{\diamond})$. Denote its spectral data by $\{ \la_n^{\diamond}, \alpha_n^{\diamond}\}_{n \ge 1}$. Furthermore, consider the data $\{ \la_n^K, \alpha_n^K\}_{n \ge 1}$ and the corresponding functions $\sigma^K(x)$, $r_1^K(\la)$, and $r_2^K(\la)$
for all sufficiently large $K$. By virtue of Lemma~\ref{coeffscon}, the coefficients $(\sigma^K, r_1^K, r_2^K)$ converge to $(\sigma^{\diamond}, r_1^{\diamond}, r_2^{\diamond})$ as $K \to \infty$. Then, Lemma~\ref{spectraldatacon} implies that $\{ \la_n^K, \alpha_n^K\} \to \{ \la_n^{\diamond}, \alpha_n^{\diamond}\}$, $K \to \infty$. On the other hand, it follows from \eqref{sdK} that $\{ \la_n^K, \alpha_n^K\} \to \{ \la_n, \alpha_n\}$, $K \to \infty$. Hence $\la_n = \la_n^{\diamond}$ and $\alpha_n = \alpha_n^{\diamond}$ for all $n \ge 1$. In view of the uniqueness theorem (Theorem~\ref{thm:uniq}), we conclude that $\sigma(x)= \sigma^{\diamond}(x)$ a.e. on $(0,\pi)$ and $r_j(\la) \equiv r_j^{\diamond}(\la)$, $j = 1, 2$. This implies the validity of the reconstruction formulas \eqref{sigma_series}--\eqref{r2_series} for the coefficients of $\mathcal L$.
\end{proof}

\section{Solution of the inverse problem} \label{sec:solution_alg}

In this section, we provide a reconstruction procedure for solving Inverse Problem~\ref{ip:main} and illustrate it by a simple example.
The proof method of the Theorem~\ref{mainthm} implies the following constructive algorithm.

\begin{alg}  \label{alg:sing}
Let $M^1(\la)$, $p_1(\la)$, and $p_2(\la)$ be given and, for definiteness, $M_1 = M_2$. We have to find $\sigma(x)$, $r_1(\la)$, and $r_2(\la)$.
\begin{enumerate} 
  \item Find $M(\la)$ using \eqref{pass}: $M(\la) = \dfrac{p_1(\la)M^1(\la)}{1+p_2(\la)M^1(\la)}$.
  \item Find $\{ \la_n, \alpha_n\}_{n \ge 1}$ from $M(\la)$: $\{\la_n\}_{n \ge 1}$ are the poles of $M(\la)$ and 
  \begin{gather*}
  \alpha_{k+j} = \mathop{\mathrm{Res}}\limits_{\la = \la_{k}} M(\la)(\la - \la_k)^j, \quad k \in I, \quad j = \overline{0, m_k-1}.
  \end{gather*}
  \item Find $M_1$ using the asymptotics \eqref{la_asymp} of $\{ \la_n\}_{n \ge 1}$: 
  \begin{equation} \label{findM1}
  	M_1 = \lim_{n \to \infty} (n - 1 - \sqrt{\la_n}).
  \end{equation}
  \item Take the model problem \eqref{eqv_tilde} and find $\{\tilde \la_n, \tilde \alpha_n\}_{n \ge 1}$ by \eqref{tildesd}, $\tilde\varphi(x, \la) = \cos \rho x$, $\tilde M(\la)$ and $\tilde \Phi(x, \la)$ by \eqref{tildeMmain}. 
  \item Find $\tilde A(x, \la, \mu)$ by \eqref{funcA}, the functions 
\begin{equation} \label{findA}	  
  \tilde A_{k+j,i} (x, \la) = \mathop{\mathrm{Res}}\limits_{\mu = \la_{ki}} \tilde A(x, \la, \mu)(\mu - \la_{ki})^j, \quad k \in I_i, \quad j = \overline{0, m_k-1}, \quad i = 0, 1,
\end{equation}  
  and $\tilde Q_{n+p, i; k, j} (x) = \frac{1}{p!} \frac{\partial^p}{\partial \la^p} \tilde A_{k, j}(x, \la)\bigg|_{\la=\la_{n, i}}, \quad n \in I_i, \quad p=\overline{0, m_{ni}-1}, \quad i=0, 1.$.
  \item Build $\tilde \psi(x) = (\tilde \psi_v(x))_{v \in V}$ using \eqref{psimain} and the operator $\tilde H(x) = (\tilde H_{u,v}(x))_{u,v \in V}$ using \eqref{Hmatr}.
  \item Solve the main equation \eqref{main_equation} and so find $\psi(x) = (\psi_v(x))_{v \in V}$.
  \item Get $\varphi_{n,i}(x)$ from $\psi_{ni}(x)$, $n \ge 1$, $i=0, 1$:
  \begin{gather*}
  \varphi_{n,1}(x) = \psi_{n1}(x), \quad
  \varphi_{n,0}(x) = \xi_n \psi_{n0}(x) + \psi_{n1}(x).
  \end{gather*}
  \item Build the function $\varphi(x, \la)$ using \eqref{varphifromtilde}.
  \item Find $\sigma(x)$, $r_1(\la)$, and $r_2(\la)$ using \eqref{sigma_series}--\eqref{r2_series}.
\end{enumerate}
\end{alg}

Algorithm~\ref{alg:sing} is theoretical. Relying on it, one can study existence and  stability of the inverse problem solution as well as develop a numerical method for reconstruction. However, these issues require a separate research. Note that, in practice, only a finite amount of spectral data $\{ \la_n, \alpha_n\}_{n = 1}^K$ are available. Therefore, one can use the finite linear system \eqref{mainK} instead of the main equation \eqref{main_equation} and then apply the formulas \eqref{sigmaK_series}, \eqref{r1K_series}, and \eqref{r2K_series} to find an approximate solution $(\sigma^K, r_1^K, r_2^K)$. Furthermore, it is worth noting that, in a numerical algorithm, it is not necessary to explicitly solve the main equation. One can use the iterative technique of \cite{IgYu}.

\smallskip

\textbf{Example.} Suppose that $\alpha_2$ and $\lambda_2$ are some given non-zero complex numbers, $\la_2 \ne n^2$, $n \in \mathbb N$.
Let $p_1(\la)=1$, $p_2(\la)=0$, $M^1(\la) = \dfrac{\alpha_2}{\la-\la_2}+\dfrac{\cos \rho \pi}{\rho \sin\rho\pi}$.
 Then \eqref{pass} implies
$$
M(\la) = \dfrac{\alpha_2}{\la-\la_2}+\dfrac{\cos \rho \pi}{\rho \sin\rho\pi} = \dfrac{1}{\pi\la} + \dfrac{\alpha_2}{\la-\la_2}+\dfrac{2}{\pi}\sum_{n=3}^{\infty}\dfrac{1}{\la-(n-2)^2}.
$$

Due to Lemma~\ref{weyl_func}, the eigenvalues and the weight numbers equal
$$
\la_1=0, \quad \la_2, \quad \la_n = (n-2)^2, \, n \ge 3, \quad \text{and} \quad
\alpha_1=\dfrac{1}{\pi}, \quad \alpha_2, \quad \alpha_n = \dfrac{2}{\pi}, \, n \ge 3,
$$
respectively. In other words, for $n \in \mathbb N \setminus \{ 2\}$, the spectral data $\{ \la_n, \alpha_n\}$ coincide with the spectral data $\{ \tilde \la_n, \tilde \alpha_n\}$ of the model problem (see \eqref{tildesd}) and, for $n = 2$, with the given complex numbers $\{ \la_2, \alpha_2\}$.

By formula \eqref{findM1}, we find $M_1 = 1$.
Next, we can construct the function $\tilde A(x, \la, \mu)$ by \eqref{funcA}:
$$
\tilde A(x, \la, \mu) = \dfrac{\alpha_2}{\mu-\la_2}\int_0^x {\tilde \varphi(t, \la) \tilde \varphi(t, \mu)} dt.
$$

Recall that $\tilde \varphi(t, \la) = \cos(\sqrt \la t)$.
As $\tilde A(x, \la, \mu)$ has the only pole $\mu = \la_2$, then \eqref{findA} implies
\begin{equation} \label{A20}
\def\arraystretch{1.5}
\begin{array}{l}
\tilde A_{2,0} (x, \la)  = \alpha_2\int_0^x \cos (\sqrt \la t) \cos(\sqrt{\la_2} t) dt, \\
\tilde A_{k,i} (x, \la)  = 0, \quad (k,i) \ne (2,0).
\end{array}
\end{equation}

Furthermore, since $\la_{n0} = \la_{n1}$ and $\alpha_{n0} = \alpha_{n1}$ for $n \ne 2$, one can easily check that the summation terms for $k \ne 2$ in \eqref{series} and \eqref{sigma_series} vanish. Therefore, instead of the infinite system \eqref{series}, we can consider the following system of size $(2 \times 2)$:
\begin{equation} \label{sys2}
\begin{cases}
\tilde \varphi_{2,0}(x)  = \varphi_{2,0}(x) + \tilde Q_{2,0; 2,0}(x) \varphi_{2,0}(x) - \tilde Q_{2,0; 2,1}(x) \varphi_{2,1}(x), \\
\tilde \varphi_{2,1}(x) = \varphi_{2,1}(x) + \tilde Q_{2,1; 2,0}(x) \varphi_{2,0}(x) - \tilde Q_{2,1; 2,1}(x) \varphi_{2,1}(x).
\end{cases}
\end{equation}

Note that, in this simple case, we do not transform the system \eqref{sys2} to the main equation \eqref{main_equation}, because the series in \eqref{series} turn into finite sums and we do not need to analyze their convergence. Using \eqref{A20}, we find the entries $\tilde Q_{n,i; k,j}(x)$ used in \eqref{sys2}:
\begin{align*}
\tilde Q_{2, 0; 2, 0}(x) & = \dfrac{\alpha_2}{4}\Bigr(\dfrac{\sin (2 \sqrt{\la_2}x)}{\sqrt{\la_2}} + 2x\Bigl), \\
\tilde Q_{2, 1; 2, 0}(x) & = \dfrac{\alpha_2}{\la_2}\Bigr(\dfrac{\sin (\sqrt{\la_2}x)}{\sqrt{\la_2}} - \dfrac{1}{2}\sqrt{\la_2}x^2\sin(\sqrt{\la_2}x) - x\cos(\sqrt{\la_2}x)\Bigl), \\
\tilde Q_{2,1; 2,0}(x) & = \tilde Q_{2,1; 2,1}(x) = 0.
\end{align*}

Solve the main equations \eqref{sys2}:
\begin{align} \label{varphi20}
\varphi_{2, 0}(x) & = \dfrac{\cos(\sqrt{\la_2}x)}{1+\dfrac{\alpha_2}{4}\Bigr(\dfrac{\sin (2 \sqrt{\la_2}x)}{\sqrt{\la_2}} + 2x\Bigl)}, \\ \nonumber
\varphi_{2, 1}(x) & = 1 - \tilde Q_{2, 1; 2, 0}(x)\varphi_{2, 0}(x).
\end{align}

Clearly, the solution is well-defined if and only if
\begin{equation} \label{cond}
1+\dfrac{\alpha_2}{4}\Bigr(\dfrac{\sin (2 \sqrt{\la_2}x)}{\sqrt{\la_2}} + 2x\Bigl) \ne 0
\end{equation}
for all $x \in [0,\pi]$. We have to impose this requirement on the given numbers $\la_2$ and $\alpha_2$. Note that, for each complex $\la_2 \ne 0$, the relation \eqref{cond} holds for all complex $\alpha_2 \ne 0$ with sufficiently small $|\alpha_2|$. 

Then, we find $\sigma(x)$ by \eqref{sigma_series}:
\begin{align} \nonumber
\sigma(x) & = - 2\sum_{j = 0}^1 (-1)^j\alpha_{2j}\Bigg(\tilde \varphi_{2,j}(x)\varphi_{2,j}(x) - \frac{1}{2}\Bigg) \\ \label{sigma-ex} & = \alpha_2 - \dfrac{8\alpha_2\sqrt{\la_2}\cos^2(\sqrt{\la_2}x)}{4\sqrt{\la_2}+\alpha_2\sin(2\sqrt{\la_2}x)+2x\sqrt{\la_2}}.
\end{align}

Next, consider the reconstruction formulas \eqref{r1_series} and \eqref{r2_series} for the polynomials. The integral parts vanish, since all the eigenvalues $\{ \la_n\}$ and $\{ \tilde \la_n\}$ are simple. (Although $\tilde \la_1 = \tilde \la_2 = 0$, we have $\tilde \alpha_2 = 0$, so, in fact, $\la = 0$ is a simple pole of $\tilde M(\la)$). Hence, one can take $N = 0$. Recall that $\la_{k0} = \la_{k1}$ for $k \ne 2$. Then, $\tilde \varphi_{k,0}'(\pi) = \tilde \varphi_{k,1}'(\pi) = 0$, $k \ne 2$. Consequently, the reconstruction formulas \eqref{r1_series} and \eqref{r2_series} turn into
\begin{align*}
r_1(\la) & = (\la - \la_{20})\Bigg(1 - \frac{\alpha_{20}\tilde\varphi'_{2,0}(\pi)\varphi_{2,0}(\pi)}{\la - \la_{20}} \Bigg), \\
r_2(\la) & = (\la - \la_{20}) \Bigg(\frac{\alpha_{20}\tilde \varphi'_{2,0} (\pi) \varphi^{[1]}_{2,0} (\pi)}{\la-\la_{20}} - \sum_{j=0}^1 (-1)^j\alpha_{2j}(\tilde \varphi_{2,j} (\pi) \varphi_{2,j} (\pi) - 1)\Bigg).
\end{align*}

Calculations show that
\begin{align*}
r_1(\la) = & \la - \la_2 + \dfrac{2\alpha_2\la_2\sin(2\sqrt{\la_2}\pi)}{4\sqrt{\la_2}+\alpha_2\sin(2\sqrt{\la_2}\pi)+2\pi\sqrt{\la_2}}, \\
r_2(\la) = & \alpha_2\la\bigl(1-\varphi_{2,0}(\pi)\cos(\sqrt{\la_2}\pi)\bigr) + \alpha_2\bigl(\la_2\varphi_{2,0}(\pi)\cos(\sqrt{\la_2}\pi)-\la_2 \\ & +\varphi'_{2,0}(\pi)\sin(\sqrt{\la_2}\pi)-\sigma(\pi)\varphi_{2,0}(\pi)\sin(\sqrt{\la_2}\pi)\bigr),
\end{align*}
where $\varphi_{2,0}(x)$ and $\sigma(x)$ are given in \eqref{varphi20} and \eqref{sigma-ex}, respectively. Thus, we have recovered the polynomials of degree $1$.

\section{Regular potential} \label{sec:reg}

The reconstruction formulas of Theorem~\ref{mainthm} are novel even for the case of regular potential $q \in L_2(0,\pi)$, which was studied by Freiling and Yurko~\cite{FrYu}. Therefore, in this section, we consider the problem statement from \cite{FrYu} and apply our results to this inverse problem.

Let us introduce the boundary problem $\mathscr {L}^1 = \mathscr {L}^1(q, p_1, p_2, \check r_1, \check r_2)$:
\begin{gather} \label{eqvreg}
    -y''(x) + q(x)y(x) = \la y(x), \quad x \in (0, \pi), \\ \label{bcreg}
    p_1(\la)y'(0) - p_2(\la)y(0) = 0, \quad \check r_1(\la)y'(\pi) + \check r_2(\la)y(\pi) = 0,
\end{gather}
were $q$ complex-valued potential of $L_2(0, \pi)$, the boundary condition \eqref{bcreg} at $x = 0$ contains relatively prime polynomials $p_j(\la)$, $j = 1, 2$, which can be represented in the form \eqref{polsp}, and the BC at $x = \pi$, relatively prime polynomials $\check r_j(\la)$, $j = 1, 2$, which can be represented in the form \eqref{polsr}. We also suppose that $N_1 = N_2$, $M_1 = M_2$, $a_{N_1} = c_{M_1} = 1$, and $b_{N_2} \ne 0$.

Let us define the function $\psi(x, \la)$ as the solution of equation \eqref{eqvreg} satisfying the initial conditions $\psi(\pi, \la) = \check r_1(\la)$, $\psi'(\pi, \la) = -\check r_2(\la)$.
It can be easily seen that the eigenvalues of the problem $\mathscr {L}^1$ coincide with the zeros of the entire characteristic function
\begin{gather*}
    \Delta^1(\la) = p_1(\la) \psi'(0, \la) - p_2(\la)\psi'(0, \la).
\end{gather*}

Let us introduce the auxiliary problem $\mathscr {L}_0 = \mathscr {L}_0(q, \check r_1, \check r_2)$:
\begin{gather*}
    -y''(x) + q(x)y(x) = \la y(x), \quad x \in (0, \pi), \\
    y(0) = 0, \quad \check r_1(\la)y'(\pi) + \check r_2(\la)y(\pi) = 0.
\end{gather*}

The characteristic function of this problem has the form $\Delta_0(\la) = \psi(0, \la)$.
Let us introduce the Weyl function $M^1(\la)$:
\begin{gather*}
    M^1(\la) = -\frac{\Delta_0(\la)}{\Delta^1(\la)}.
\end{gather*}

Note that Freiling and Yurko in \cite{FrYu} assumed that the zeros of the polynomials $p_1(\la)$ and $p_2(\la)$ are known a priori. It means that the polynomial $p_1(\la)$ is given (its leading coefficient equals $1$) and the polynomial $p_2(\la)$ can be easily found if we recover its leading coefficient. Therefore, the inverse problem of \cite{FrYu} can be formulated as follows.

\begin{ip} \label{ip:mainreg}
Given the Weyl function $M^1(\la)$, the polynomial $p_1(\la)$ and the zeros $\{ z_j \}_{j = 1}^{N_2}$ of $p_2(\la)$, find $q(x)$, $\check r_1(\la)$, $\check r_2(\la)$, and $b_{N_2}$.
\end{ip}

We can find the coefficient $b_{N_2}$ from the  asymptotics for $M^1(\la)$ (see \cite{FrYu}):
\begin{gather} \label{bN2}
M^1(\la) = \frac{1}{i\rho\la^{N_1}}\bigg(1-\frac{b_{N_2}}{i\rho} + o(\rho^{-1})\bigg).
\end{gather}

As we know $M^1(\la)$, $p_1(\la)$, $p_2(\la)$, then we can find $M(\la)$ by \eqref{pass} for problem $\mathscr {L} = \mathscr {L}(q, \check r_1, \check r_2)$:
\begin{gather*}
    -y''(x) + q(x)y(x) = \la y(x), \quad x \in (0, \pi), \\
    y'(0) = 0, \quad \check r_1(\la)y'(\pi) + \check r_2(\la)y(\pi) = 0,
\end{gather*}

This problem is equivalent to the next one:
\begin{gather} \label{eqv_sing}
    -(y^{[1]})'(x) - \sigma(x)y^{[1]}(x) - \sigma^2(x)y(x) = \la y(x), \quad x \in (0, \pi), \\ \label{bc_sing}
    y^{[1]}(0) = 0, \quad r_1(\la)y^{[1]}(\pi) + r_2(\la)y(\pi) = 0,
\end{gather}
where $\sigma(x) = \int_0^x q(t) \, dt$, $\sigma \in W_2^1[0,\pi]$, $\sigma(0) = 0$, $r_1(\la) = \check r_1(\la)$, and $r_2(\la) = \check r_2(\la) + \sigma(\pi)$. Obviously, $\deg(r_2) = \deg(\check r_2)$.

Thus, we arrive at the following algorithm for solving Inverse Problem~\ref{ip:mainreg}.

\begin{alg} \label{alg:reg}
Let the Weyl function $M^1(\la)$, the polynomial $p_1(\la)$ and the zeros $\{z_j\}_{j=1}^{N_2}$ be given. We have to find $\sigma(x)$, $r_1(\la)$, $r_2(\la)$, and $b_{N_2}$.
\begin{enumerate}
  \item Find $b_{N_2}$ using \eqref{bN2}.
  \item Construct $p_2(\la) = b_{N_2} \prod\limits_{j = 1}^{N_2} (z - z_j)$.
  \item Pass to the problem \eqref{eqv_sing}--\eqref{bc_sing} by finding $M(\la)$ using \eqref{pass}.
  \item By using $M(\la)$, build $\sigma(x)$, $r_1(\la)$, and $r_2(\la)$ following the steps 2-10 of Algorithm~\ref{alg:sing}.
  \item Find $q(x) := \sigma'(x)$, $\check r_1(\la) := r_1(\la)$, and $\check r_2(\la) := r_2(\la) - \sigma(\pi)$.
\end{enumerate}
\end{alg}

In particular, Algorithm~\ref{alg:reg} implies the following formula for $q(x)$:
\begin{align} \notag 
q(x) &= -\frac{1}{\pi i}\oint_{\Gamma_N} {\Bigg(\tilde \varphi'(x, \mu)\varphi(x, \mu) + \tilde \varphi(x, \mu)\varphi'(x, \mu)\Bigg)\hat M(\mu)}d\mu \\ \label{sigma_series_reg}
& - 2 \sum_{k=N+1}^\infty \sum_{j=0}^1 (-1)^j\alpha_{kj}(\tilde \varphi'_{k,j}(x)\varphi_{k,j}(x) + \tilde \varphi_{k,j}(x)\varphi'_{k,j}(x)).
\end{align}

It can be shown that the series in \eqref{sigma_series_reg} converges in $L_2(0,\pi)$.
Calculation of the contour integral by the Residue theorem shows that \eqref{sigma_series_reg} is analogous to the formula (60) of \cite{But13} for recovering the regular potential of the Sturm-Liouville problem with the Dirichlet BCs. In the case of simple eigenvalues, the formula \eqref{sigma_series_reg} coincide with the reconstruction formula (1.6.25) of the book \cite{FY01} by Freiling and Yurko for the Sturm-Liouville problem with the Robin BCs. However, for the case of the polynomials in the BCs, the formula \eqref{sigma_series_reg} is new.

Let us consider the special case $M_1 = M_2 = 0$, which corresponds to the Robin BC $y^{[1]}(\pi) + b_0 y(\pi) = 0$.
Then, in view of Remark~\ref{rem:poly}, the formulas \eqref{r1_series} and \eqref{r2_series} generate the constants $r_1(\la) = 1$ and $r_2(\la) = b_0$, respectively. In order to find the constant $b_0$, we pass to the limit as $|\la| \to \infty$ in \eqref{r2_series} and so obtain the formula
$$
b_0 = -\frac{1}{2\pi i} \oint_{\Gamma_N} (\tilde \varphi(\pi, \mu) \varphi(\pi, \mu) - 1) \hat M(\mu) \, d\mu - \sum_{k = N+1}^{\infty} (-1)^j \alpha_{kj} (\tilde \varphi_{k,j}(\pi) \varphi_{k,j}(\pi) - 1),
$$
which generalizes the formula (5.15) from \cite{BondTamkang} to the case of multiple eigenvalues. In the case of regular potential, for the recovery of $\check b_0$ in the Robin BC $y'(\pi) + \check b_0 y(\pi) = 0$, we have
$$
\check b_0 = b_0 - \sigma(\pi) = \frac{1}{2\pi i} \oint_{\Gamma_N} \tilde \varphi(\pi, \mu) \varphi(\pi, \mu) \, d\mu  + \sum_{j = 0}^1 \sum_{k = N+1}^{\infty} (-1)^j \alpha_{kj} \tilde \varphi_{k,j}(\pi) \varphi_{k,j}(\pi).
$$

This formula generalizes the relation (1.6.25) from \cite{FY01} for the recovery of the Robin BC coefficient $H$ to the case of multiple eigenvalues.
Thus, our reconstruction formulas generalize the previously known results for constant coefficients in the BCs, and they are novel for the case of polynomials.

It is worth mentioning that Freiling and Yurko \cite{FrYu} suggested to recover the unknown coefficients as follows:
\begin{equation} \label{recFY}
q(x) = \la + \frac{\varphi''(x, \la)}{\varphi(x, \la)}, \qquad \frac{r_2(\la)}{r_1(\la)} = -\frac{\Phi'(\pi, \la)}{\Phi(\pi, \la)},
\end{equation}
where $\varphi(x, \la)$ and $\Phi(x, \la)$ were previously found by using the analogs of the formulas \eqref{varphifromtilde} and \eqref{Phifromtilde}, respectively. However, our reconstruction formulas \eqref{sigma_series}--\eqref{r2_series} are more convenient than \eqref{recFY} for investigation of the inverse problem solvability and stability, because our formulas guarantee that the resulting function $\sigma(x)$ belongs to $L_2(0,\pi)$ and $r_j(\la)$, $j = 1,2$, are polynomials of desired degrees (see Remark~\ref{rem:poly}). 

\medskip

\noindent\textbf{Funding}: This work was supported by Grant 21-71-10001 of the Russian Science Foundation, https://rscf.ru/en/project/21-71-10001/.

\noindent Egor Evgenevich Chitorkin \\
1. Institute of IT and Cybernetics, Samara National Research University, \\
Moskovskoye Shosse 34, Samara 443086, Russia, \\
2. Department of Mechanics and Mathematics, Saratov State University, \\
Astrakhanskaya 83, Saratov 410012, Russia, \\
e-mail: {\it chitorkin.ee@ssau.ru} \\

\noindent Natalia Pavlovna Bondarenko \\
1. Department of Mechanics and Mathematics, Saratov State University, \\
Astrakhanskaya 83, Saratov 410012, Russia, \\
2. Department of Applied Mathematics and Physics, Samara National Research University, \\
Moskovskoye Shosse 34, Samara 443086, Russia, \\
3. Peoples' Friendship University of Russia (RUDN University), \\
6 Miklukho-Maklaya Street, Moscow, 117198, Russia, \\
e-mail: {\it bondarenkonp@info.sgu.ru}

\end{document}